\documentclass[10pt,reqno]{amsart}
\usepackage{hyperref}
\usepackage{latexsym,amsmath}
\usepackage{enumerate}
\usepackage{amsfonts}
\usepackage{amssymb}
\usepackage{latexsym}
\hypersetup{colorlinks=true,citecolor=blue, linkcolor=blue,
urlcolor=blue}

\usepackage{graphicx}
\usepackage[backend=biber, style=trad-abbrv]{biblatex}
\DeclareFieldFormat{titlecase}{#1}

\usepackage[T1]{fontenc}
\usepackage{bbm}
\usepackage{color}
\usepackage{fullpage}

\numberwithin{equation}{section}

\newcommand{\Fplus}{\cF^+_{d,\beta}}
\newcommand{\Fone}{\cF^1_{d,\beta}}
\newcommand{\Fix}{\cF_{d,\beta}}
\newcommand\fix{\cF_{d,\beta}}
\newcommand\uniq{{\beta_u}}
\newcommand\curie{{\beta_p}}
\newcommand\KS{{\beta_h}}
\newcommand\fSferro{\fS_{\mathrm{f}}}
\newcommand\fSpara{\fS_{\mathrm{p}}}
\newcommand\para{{\mathrm{p}}}
\newcommand\ferro{{\mathrm{f}}}

\newcommand\Zferro{Z_{\mathrm{f}}}
\newcommand\Zpara{Z_{\mathrm{p}}}
\newcommand\Yferro{Y_{\mathrm{f}}}
\newcommand\Ypara{Y_{\mathrm{p}}}
\newcommand\Gferro{\hat\G_{\mathrm{f}}}

\newcommand\Gpara{\hat\G_{\mathrm{p}}}
\newcommand\sferro{\hat\SIGMA_{\mathrm{f}}}
\newcommand\spara{\hat\SIGMA_{\mathrm{p}}}

\newcommand\sparaGpara{\SIGMA_{\Gpara,\mathrm{p}}}
\newcommand\sferroGferro{\SIGMA_{\Gferro,\mathrm{f}}}
\newcommand\nuferro{\nu_{\mathrm{f}}}
\newcommand\rhoferro{\rho_{\mathrm{f}}}
\newcommand\muferro{\mu_{\mathrm{f}}}
\newcommand\nupara{\nu_{\mathrm{p}}}
\newcommand\mupara{\mu_{\mathrm{p}}}
\newcommand\rhopara{\rho_{\mathrm{p}}}
\newcommand\Spara{S_{\mathrm{p}}}
\newcommand\Sferro{S_{\mathrm{f}}}
\newcommand\rferro{r_{\mathrm{f}}}
\newcommand\chiferro{\chi_{\mathrm{f}}}
\newcommand\phiferro{\phi_{\mathrm{f}}}

\newcommand\sferrod{\sferro(\eps)}
\newcommand\sparad{\spara(\eps)}
\newcommand\Sparad{\Spara(\eps)}
\newcommand\Sferrod{\Sferro(\eps)}
\newcommand\tSferrod{\tilde\Sferro(\eps)}
\newcommand\Sparadp{\Spara(\eps')}
\newcommand\Sferrodp{\Sferro(\eps')}
\newcommand\tSferrodp{\tilde\Sferro(\eps')}
\newcommand\emm{\mathrm{e}}
\newcommand\Zferrod{\Zferro^{\eps}}
\newcommand\Zparad{\Zpara^{\eps}}
\newcommand\Zferrodp{\Zferro^{\eps'}}
\newcommand\Zparadp{\Zpara^{\eps'}}

\newcommand\vsigma{\vec\sigma}

\renewcommand{\vec}[1]{\boldsymbol{#1}}

\renewcommand{\subset}{\subseteq}

\newcommand\disteq{\,\stacksign{d}=\,}

\newcommand\vX{\vec X}

\newcommand{\GG}{\mathbb{G}}
\newcommand{\TT}{\mathbb T}

\newcommand\nix{\,\cdot\,}

\newcommand\hG{\hat\G}

\newcommand\G{\vec G}

\newcommand\KL[2]{D_{\mathrm{KL}}\bc{{{#1}\|{#2}}}}
\newcommand\SIGMA{\vec\sigma}

\newcommand\fS{\mathfrak{S}}

\newcommand\cB{\mathcal{B}}
\newcommand\cC{\mathcal{C}}

\newcommand\cF{\mathcal{F}}
\newcommand\cG{\mathcal{G}}
\newcommand\cE{\mathcal{E}}

\newcommand\cQ{\mathcal{Q}}
\newcommand\cH{\mathcal{H}}

\newcommand\cP{\mathcal{P}}

\newcommand\cZ{\mathcal{Z}}
\def\cR{{\mathcal R}}
\def\cC{{\mathcal C}}
\def\cE{{\mathcal E}}

\newcommand\tG{\tilde\G}
\newcommand\tp{\tilde p}
\newcommand\tr{\tilde r}

\newcommand\tn{\tilde n}

\newcommand\eul{\mathrm{e}}
\newcommand\eps{\varepsilon}

\newcommand\Erw{\mathbb{E}}
\newcommand\ex{\Erw}

\newcommand{\vecone}{\vec{1}}

\newcommand{\escf}{\mathrm{f}}

\newcommand\dTV{d_{\mathrm{TV}}}

\newcommand\bc[1]{\left({#1}\right)}
\newcommand\cbc[1]{\left\{{#1}\right\}}

\newcommand\brk[1]{\left\lbrack{#1}\right\rbrack}

\newcommand\norm[1]{\left\|{#1}\right\|}
\newcommand\abs[1]{\left|{#1}\right|}

\newcommand{\whp}{w.h.p.}

\newcommand{\stacksign}[2]{{\stackrel{\mbox{\scriptsize #1}}{#2}}}
\newcommand{\tensor}{\otimes}

\newcommand\pr{\mathbb{P}} 
\renewcommand\Pr{\pr} 
\newcommand\Lem{Lemma}
\newcommand\Prop{Proposition}
\newcommand\Thm{Theorem}

\newcommand\Cor{Corollary}

\newtheorem{definition}{Definition}[section]

\newtheorem{theorem}[definition]{Theorem}
\newtheorem{lemma}[definition]{Lemma}
\newtheorem{proposition}[definition]{Proposition}
\newtheorem{corollary}[definition]{Corollary}

\newtheorem*{Lemmagiantperc}{Lemma~\ref{lem:giantperc}}
\newtheorem*{PropSWpara}{Proposition~\ref{thm:SWpara}}
\newtheorem*{PropSWferro}{Proposition~\ref{thm:SWferro}}
\newtheorem*{Lembottle}{Lemma~\ref{lem_bottleNeck}}

\begin{document}

\title{Metastability of the Potts ferromagnet \\on random regular graphs}

\author{Amin Coja-Oghlan, Andreas Galanis, Leslie Ann Goldberg, Jean Bernoulli Ravelomanana, Daniel \v{S}tefankovi\v{c}, Eric~Vigoda}
\thanks{A preliminary version of the manuscript (without the proofs) appeared in the proceedings of \emph{ICALP 2022} (Track A).\\
For the purpose of Open Access, the author has applied a CC BY public copyright licence to any Author Accepted Manuscript version arising from this submission. All data is provided in full in the results section of this paper.}

\address{Amin Coja-Oghlan, {\tt amin.coja-oghlan@tu-dortmund.de}, TU Dortmund, Faculty of Computer Science, 12 Otto Hahn St, Dortmund 44227, Germany}

\address{Andreas Galanis, {\tt andreas.galanis@cs.ox.ac.uk}, University of Oxford, Department of Computer Science, Wolfson Bldg, Parks Rd, Oxford OX1 3QD, UK}

\address{Leslie Ann Goldberg, {\tt leslie.goldberg@seh.ox.ac.uk}, University of Oxford, Department of Computer Science, Wolfson Bldg, Parks Rd, Oxford OX1 3QD, UK}

\address{Jean Bernoulli Ravelomanana, {\tt jean.ravelomanana@epfl.ch}, \'{E}cole Polytechnique F\'{e}d\'{e}rale de Lausanne, INR 139, Lausanne CH-1015, CH}

\address{Daniel \v{S}tefankovi\v{c}, {\tt stefanko@cs.rochester.edu}, Univerity of Rochester, Department of Computer Science, 2315 Wegmans Hall, Rochester NY 14627, USA}

\address{Eric Vigoda, {\tt vigoda@ucsb.edu}, UC Santa Barbara, Computer Science, 2104 Harold Frank Hall, Santa Barbara CA 93106, USA}

\begin{abstract}
We study the performance of Markov chains for the $q$-state ferromagnetic Potts model on random regular graphs. While the cases of the grid and the complete graph are by now well-understood, the case of random regular graphs has resisted a detailed analysis and, in fact, even analysing the properties of the Potts distribution has remained elusive. It is conjectured that the performance of Markov chains is dictated by metastability phenomena, i.e., the presence of ``phases'' (clusters) in the sample space where Markov chains with local update rules, such as the Glauber dynamics,  are bound to take exponential time to escape, and therefore cause slow mixing. The phases that are believed to drive these metastability phenomena in the case of the Potts model emerge as local, rather than global, maxima of the so-called Bethe functional, and previous  approaches of analysing these phases based on  optimisation arguments fall short of the task.

Our first contribution is to detail the emergence of the two relevant phases for the $q$-state Potts model on the $d$-regular random graph for all integers $q,d\geq 3$, and establish that for an interval of temperatures, delineated by the uniqueness and a broadcasting threshold on the $d$-regular tree, the two phases coexist (as possible metastable states). The proofs are based on a conceptual connection between spatial  properties and the structure of the Potts distribution on the random regular graph, rather than complicated moment calculations.   This significantly refines earlier results by Helmuth, Jenssen, and Perkins who had established phase coexistence for a small interval around the so-called ordered-disordered threshold (via different arguments) that applied for large $q$ and $d\geq 5$. 

Based on our new structural understanding of the model, our second contribution is to obtain metastability results for two classical Markov chains for the Potts model. We first  complement recent fast mixing results for Glauber dynamics by Blanca and Gheissari below the uniqueness threshold, by showing an exponential lower bound on the mixing time above the uniqueness threshold. Then, we obtain tight results even for the non-local and more elaborate Swendsen-Wang chain, where we establish slow mixing/metastability for the whole interval of temperatures where the chain is conjectured to mix slowly on the random regular graph. The key is  to bound the conductance of the chains using a random graph ``planting'' argument combined with delicate bounds on random-graph percolation.
\end{abstract}

\maketitle
\section{Introduction}\label{Sec_intro}

\subsection{Motivation}
Spin systems on random graphs have turned out to be a source of extremely challenging problems at the junction of mathematical physics and combinatorics~\cite{MP1,MP2}.
Beyond the initial motivation of modelling disordered systems, applications have sprung up in areas as diverse as computational complexity, coding theory, machine learning and even screening for infectious diseases; e.g.~\cite{Abbe_2017,GroupTesting,Andreas,Mezard,RichardsonUrbanke,Sly,SlySun}.
Progress has been inspired largely by techniques from statistical physics, which to a significant extent still await a rigorous justification.
The physicists' sophisticated but largely heuristic tool is the Belief Propagation message passing scheme in combination with a functional called the Bethe free energy~\cite{MM}.
Roughly speaking, the fixed points of Belief Propagation are conjectured to correspond to the `pure states' of the underlying distribution, with the Bethe functional gauging the relative weight of the different pure states.
Yet at closer inspection matters are actually rather complicated.
For instance, the system typically possesses spurious Belief Propagation fixed points without any actual combinatorial meaning, while other fixed points need not correspond to metastable states that attract dynamics such as the Glauber Markov chain~\cite{BinomialMatrix,CKPZ}.
Generally, the mathematical understanding of the connection between Belief Propagation and dynamics leaves much to be desired.

In this paper we investigate the ferromagnetic Potts model on the random regular graph. Recall, for an integer $q\geq 3$ and real $\beta>0$, the Potts model on a graph $G=(V,E)$ corresponds to a probability distribution $\mu_{G,\beta}$ over all possible configurations $[q]^V$, commonly referred to as the Boltzmann/Gibbs distribution; the weight of a configuration $\sigma$ in the distribution is defined as $\mu_{G,\beta}(\sigma)=\emm^{\beta \cH_G(\sigma)}/Z_\beta(G)$ where $\cH_G(\sigma)$ is the number of edges that are monochromatic under $\sigma$, and $Z_\beta(G)= \sum_{\tau\in [q]^V}\emm^{\beta \cH_G(\tau)}$ is the normalising factor of the distribution. In physics jargon, $\beta$ corresponds to the so-called inverse-temperature of the model, $\cH_G(\nix)$ is known as the Hamiltonian, and $Z_\beta(\nix)$ is the partition function. Note, since $\beta>0$, the Boltzmann distribution assigns greater weight to configurations $\sigma$ where many edges join vertices of the same colour; thus, the pairwise interactions between vertices are ferromagnetic.

The Potts model on the $d$-regular random graph has two distinctive features. 
First, the local geometry of the random regular graph is essentially deterministic.
For any fixed radius $\ell$, the depth-$\ell$ neighbourhood of all but a tiny number of vertices is just a $d$-regular tree. Second, the ferromagnetic nature of the model precludes replica symmetry breaking, a complex type of long-range correlations~\cite{MM}. Given these, it is conjectured that the model on the random regular graph has a similar behaviour to that on the clique (the so-called mean field case), and there has already  been some preliminary evidence of this correspondence~\cite{BarbierChanMacris,DemboMontanariSun,Dembo_2014,Andreas, HJP}. On the clique,  the phase transitions are driven by a battle between two subsets of configurations (phases): (i) the paramagnetic/disordered phase, consisting of configurations where  every colour appears roughly equal number of times, and (ii) the ferromagnetic/ordered phase, where one of the colours appears more frequently than the others. It is widely believed that these two  phases also mark (qualitatively) the same type of phase transitions for the Potts model on the random regular graph, yet  this has remained largely elusive.

The main reason that this behaviour is harder to establish on the random regular graph is that it has a non-trivial global geometry which makes both the analysis of the distribution and Markov chains significantly more involved (to say the least). In particular, the emergence of the metastable states in the  distribution, which can be established by way of calculus in the mean-field case, is out of reach with single-handed analytical approaches in the random regular graph and it is therefore not surprising that it has resisted a detailed analysis so far. Likewise, the analysis of Markov chains is a far more complicated task since their evolution needs to be considered in terms of the graph geometry and therefore much harder to keep track of.

Our main contribution is to detail the emergence of the metastable states, viewed as fixed points of Belief Propagation on this model, and their connection with the dynamic evolution of the two most popular Markov chains, the Glauber dynamics and the Swensen-Wang chain. 
We prove that these natural fixed points, whose emergence is directly connected with the phase transitions of the model, have the combinatorial meaning in terms of both the pure state decomposition of the  distribution and the Glauber dynamics that physics intuition predicts they should. The proofs avoid the complicated moment calculations and the associated complex optimistion arguments that have become a hallmark of the study of spin systems on random graphs~\cite{ANP}.
Instead, building upon and extending ideas from~\cite{Victor,Greg}, we exploit a connection between spatial mixing properties on the $d$-regular tree and the Boltzmann distribution. Our metastability results for the Potts model significantly refine those appearing in the literature, especially those in \cite{Andreas, HJP} which are more relevant  to this work, see Section~\ref{sec:discussion} for a more detailed discussion.

We expect that this approach might carry over to other examples, particularly other ferromagnetic models.
Let us begin by recapitulating Belief Propagation.

\subsection{Belief Propagation}
Suppose that $n,d\geq3$ are integers such that $dn$ is even and let $\GG=\GG(n,d)$ be the random $d$-regular graph on the vertex set $[n]=\{1,\ldots,n\}$.
For an inverse temperature parameter $\beta>0$ and an integer $q\geq3$ we set out to investigate the Boltzmann distribution $\mu_{\GG,\beta}$; let us write $\vsigma_{\GG,\beta}$ for a configuration drawn from $\mu_{\GG,\beta}$.

A vital step toward understanding the Boltzmann distribution is to get a good handle on the partition function $Z_{\beta}(\GG)$.
Indeed, according to the physicsts' cavity method, Belief Propagation actually solves both problems in one fell swoop~\cite{MM}.
To elaborate, with each edge $e=uv$ of $\GG$, Belief Propagation associates two messages $\mu_{\GG,\beta,u\to v},\mu_{\GG,\beta,v\to u}$, which are probability distributions on the set $[q]$ of colours.
The message $\mu_{\GG,\beta,u\to v}(c)$ is defined as the marginal probability of $v$ receiving colour $c$ in a configuration drawn from the Potts model on the graph $\GG-u$ obtained by removing $u$.
The semantics of $\mu_{\GG,\beta,v\to u}$ is analogous.

Under the assumption that the colours of far apart vertices of $\GG$ are asymptotically independent, one can heuristically derive a set of equations that links the various messages together. For a vertex $v$, let $\partial v$ be the set of neighbours of $v$, and for an integer $\ell\geq1$ let $\partial^\ell v$ be the set of vertices at distance precisely $\ell$ from $v$.
The {\em Belief Propagation equations} read
\begin{align}\label{eqBPG}
	\mu_{\GG,\beta,v\to u}(c)&=\frac{\prod_{w\in\partial v\setminus\cbc u}1+(\eul^\beta-1)\mu_{\GG,\beta,w\to v}(c)}{\sum_{\chi\in[q]}\prod_{w\in\partial v\setminus\cbc u}1+(\eul^\beta-1)\mu_{\GG,\beta,w\to v}(\chi)}&&(uv\in E(\GG),\ c\in[q]).
\end{align}
The insight behind \eqref{eqBPG} is that once we remove $v$ from the graph, its neighbours $w\neq u$ are typically far apart from one another because $\GG$ contains only a negligible number of short cycles.
Hence, we expect that in $\GG-v$ the spins assigned to $w\in\partial v\setminus\cbc u$ are asymptotically independent.
From this assumption it is straightforward to derive the sum-product-formula \eqref{eqBPG}.

A few obvious issues spring to mind.
First, for large $\beta$ it is not actually true that far apart vertices decorrelate.
This is because at low temperature there occur $q$ different ferromagnetic pure states, one for each choice of the dominant colour.
To break the symmetry between them one could introduce a weak external field that slighly boosts a specific colour or, more bluntly, confine oneself to a conditional distribution on subspace where a specific colour dominates.
In the definition of the messages and in \eqref{eqBPG} we should thus replace the Boltzmann distribution by the conditional distribution $\mu_{\GG,\beta}(\nix\mid S)$ for a suitable $S\subset[q]^{n}$.
Second, even for the conditional measure we do not actually expect \eqref{eqBPG} to hold precisely.
This is because for any finite $n$ minute correlations between far apart vertices are bound to remain. 

Nonetheless, precise solutions $(\mu_{v\to u})_{uv\in E(\GG)}$ to \eqref{eqBPG} are still meaningful. 
They correspond to stationary points of a functional called the {\em Bethe free energy}, which connects Belief Propagation with the problem of approximating the partition function~\cite{YFW}.
Given a collection $(\mu_{u\to v})_{uv\in E(\GG)}$ of probability distributions on $[q]$, the Bethe functional reads
\begin{equation}\label{eqBetheG}
\begin{aligned}
	\cB_{\GG,\beta}\bc{(\mu_{u\to v})_{uv\in E(\GG)}}&=\frac1n\sum_{v\in V(\GG)}\log\bigg[\sum_{c\in[q]}\prod_{w\in\partial v}1+(\eul^\beta-1)\mu_{w\to v}(c)\bigg]\\&-\frac1n\sum_{vw\in E(\GG)}\log\bigg[1+(\eul^\beta-1)\sum_{c\in[q]}\mu_{v\to w}(c)\mu_{w\to v}(c)\bigg].
\end{aligned}
\end{equation}
According to the cavity method the maximum of $\cB_{\GG,\beta}\bc{(\mu_{u\to v})_{uv\in E(\GG)}}$ over all solutions $(\mu_{u\to v})_{uv\in E(\GG)}$ to \eqref{eqBPG} should be asymptotically equal to $\log Z_\beta(\GG)$ with high probability.

In summary, physics lore holds that the solutions $(\mu_{u\to v})_{uv\in E(\GG)}$ to \eqref{eqBPG} are meaningful because they correspond to a decomposition of the phase space $[q]^{n}$ into pieces where long-range correlations are absent.
Indeed, these ``pure states'' are expected to exhibit metastability, i.e., they trap dynamics such as the Glauber Markov chain for an exponential amount of time.
Moreover, the relative probabilities of the pure states are expected to be governed by their respective Bethe free energy.
In the following we undertake to investigate these claims rigorously.

Before proceeding, let us mention that ferromagnetic spin systems on random graphs have been among the first models for which predictions based on the cavity method could be verified rigorously.
Following seminal work by Dembo and Montanari on the Ising model~\cite{DemboMontanari} vindicating the ``replica symmetric ansatz'', Dembo, Montanari and Sun~\cite{DemboMontanariSun} studied, among other things, the Gibbs unique phase of the Potts ferromagnet on the random regular graph, and Dembo, Montanari, Sly and Sun~\cite{DemboMontanariSun} established the free energy of the model for all $\beta$ (and $d$ even). 
More generally, Ruozzi~\cite{Ruozzi} pointed out how graph covers~\cite{Vontobel} can be used to investigate the partition function of supermodular models, of which the Ising ferromagnet is an example.
In addition, Barbier, Chan and Macris~\cite{BarbierChanMacris} proved that ferromagnetic spin systems on random graphs are generally replica symmetric in the sense that the multi-overlaps of samples from the Boltzmann distribution concentrate on deterministic values.

\subsection{The ferromagnetic and the paramagnetic states}\label{sec:states}
An obvious attempt at constructing solutions to the Belief Propagation equations is to choose identical messages $\mu_{u\to v}$ for all edges $uv\in E(\GG)$.
Clearly, any solution $(\mu(c))_{c\in[q]}$ to the system
\begin{align}\label{eqBP}
	\mu(c)&=\frac{(1+(\eul^\beta-1)\mu(c))^{d-1}}{\sum_{\chi\in[q]}(1+(\eul^\beta-1)\mu(\chi))^{d-1}}&(c\in[q]) 
\end{align}
supplies such a `constant' solution to \eqref{eqBPG}. Let $\fix$ be the set of all solutions $(\mu(c))_{c\in[q]}$ to \eqref{eqBP}.
The Bethe functional \eqref{eqBetheG} then simplifies to
\begin{align}\label{eqBethe}
	\cB_{d,\beta}\big((\mu(c))_{c\in[q]}\big)&=\log\bigg[\sum_{c\in[q]}\bc{1+(\eul^\beta-1)\mu(c)}^d\bigg]-\frac d2\log\bigg[1+(\eul^\beta-1)\sum_{c\in[q]}\mu(c)^2\bigg].
\end{align}

One obvious solution  to \eqref{eqBP} is the uniform distribution on $[q]$; we refer to that solution as paramagnetic/disordered and denote it by $\mupara$. Apart from $\mupara$, other solutions to \eqref{eqBP} emerge as $\beta$ increases for any $d\geq3$.
Specifically, let $\uniq>0$ be the supremum value of $\beta>0$ where $\mupara$ is the unique solution to \eqref{eqBP}.\footnote{The value  does not have a closed-form expression, but there is an equivalent formulation of it given by the equality $\emm^\uniq=1+\inf_{y>1} \tfrac{(y-1)(y^{d-1}+q-1)}{y^{d-1}-y}$.}
Then, for $\beta= \uniq$, one more solution $\muferro$ emerges such that $\muferro(1)>\muferro(i)=\tfrac{1-\muferro(1)}{q-1}$ for $i=2,\ldots,q$, portending the emergence of a metastable state and, ultimately, a phase transition.
In particular, for any $\beta> \uniq$, a bit of calculus reveals there exist either one or two distinct solutions $\mu$ with  $\mu(1)>\mu(i)=\tfrac{1-\mu(1)}{q-1}$ for $i=2,\ldots,q$;   we denote by $\muferro$ the solution of \eqref{eqBP} which maximises the value $\mu(1)$ and refer to it as ferromagnetic/ordered. The value $\uniq$ is the so-called  uniqueness threshold for the Potts model on the $d$-regular tree, see, e.g., \cite{Andreas} for a more detailed discussion and related pointers.

At the critical value
\begin{align*}
	\curie&=\max\cbc{\beta\geq\uniq:\cB_{d,\beta}(\mupara)\geq\cB_{d,\beta}(\muferro)}=\log\frac{q-2}{(q-1)^{1-2/d}-1}.
\end{align*}
the ferromagnetic solution $\muferro$ takes over from the paramagnetic solution $\mupara$ as the global maximiser of the Bethe functional. For that reason, the threshold $\curie$ is also known in the literature as the ordered-disordered threshold. Yet, up to the threshold
\begin{align*}
	\KS&=\log(1+q/(d-2))
\end{align*}
the paramagnetic solution remains a local maximiser of the Bethe free energy; later, in Section~\ref{Sec_nr} we will see that $\KS$ has a natural interpretation as a tree-broadcasting threshold (and is also a conjectured threshold for uniqueness in the random-cluster representation for the Potts model, see \cite{haggstrom} for details).

The relevance of these critical values has been demonstrated in \cite{Andreas} (see also \cite{Dembo_2014} for $d$ even, and \cite{HJP} for $q$ large), where it was shown that   $\tfrac{1}{n}\log Z_\beta(\GG)$  is asymptotically equal to $\max_{\mu} \cB_{d,\beta}(\mu)$, the maximum ranging over $\mu$ satisfying \eqref{eqBP}. In particular, at the maximum it holds that $\mu=\mupara$ when $\beta<\curie$, $\mu=\muferro$ when $\beta>\curie$ and $\mu\in\{\mupara,\muferro\}$ when $\beta=\curie$.

\subsection{Slow mixing and metastability} \label{sec_Metastability} 
To investigate the two BP solutions further and obtain connections to the dynamical evolution of the model, we need to look more closely how these two solutions $\mupara,\muferro$ manifest themselves in the random regular graph. To this end, we define for a given distribution $\mu$ on $[q]$ another distribution
\begin{align}\label{eqnumu}
	\nu^\mu(c)&=\frac{(1+(\eul^\beta-1)\mu(c))^{d}}{\sum_{\chi\in[q]}(1+(\eul^\beta-1)\mu(\chi))^{d}}&(c\in[q]).
\end{align}
Let $\nuferro=\nu^{\muferro}$ and $\nupara=\nu^{\mupara}$ for brevity; of course $\nupara=\mupara$ is just the uniform distribution.
The distributions $\nuferro$ and $\nupara$ represent the expected Boltzmann marginals within the pure states corresponding to $\muferro$ and $\mupara$.
Indeed, the r.h.s.\ of \eqref{eqnumu} resembles that of \eqref{eqBP} except that the exponents read $d$ rather than $d-1$.
This means that we pass from messages, where we omit one specific endpoint of an edge from the graph, to actual marginals, where all $d$ neighbours of a vertex are present. For small $\eps>0$, it will therefore be relevant to consider the sets of configurations
\begin{align*}
	\Sferro(\eps)&=\bigg\{\sigma\in[q]^{n}:\sum_{c\in[q]}\Big|\big| \sigma^{-1}(c) \big|-n\nuferro(c)\Big|<\eps n\bigg\},\\
\Spara(\eps)&=\bigg\{\sigma\in[q]^{n}:\sum_{c\in[q]}\Big|\big| \sigma^{-1}(c) \big|-n\nupara(c)\Big|<\eps n\bigg\},
\end{align*}
whose colour statistics are about $n\nuferro$ and $n\nupara$, respectively; i.e., in $\Spara$, all colours appear with roughly equal frequency, whereas in $\Sferro$ colour 1 is favoured over the other $q-1$ colours (which appear with roughly equal frequency).

We are now in position to state our main result for Glauber dynamics.  Recall that, for a graph $G=(V,E)$, Glauber is initialised at a  configuration $\sigma_0\in [q]^{V}$; at each time step $t\geq1$, Glauber draws a vertex uniformly at random and obtains a new configuration $\sigma_t$ by updating the colour of the chosen vertex according to the conditional Boltzmann distribution given the colours of its neighbours. It is a well-known fact that Glauber converges in distribution to $\mu_{G,\beta}$; the mixing time of the chain is defined as the maximum number of steps $t$ needed to get within total variation distance $\leq 1/4$ from $\mu_{G,\beta}$, where the maximum is over the choice of the initial configuration $\sigma_0$, i.e., the quantity  $\max_{\sigma_0}\min\{t:\,\dTV( \sigma_t,\mu_{G,\beta})\leq 1/4\}$.

For metastability, we will consider Glauber launched from a random configuration from a subset $S\subset[q]^{V}$ of the state space. More precisely, let us denote by $\mu_{G,\beta,S}=\mu_{G,\beta}(\cdot \mid S)$ the conditional Boltzmann distribution on $S$. We call $S$  a {\em metastable state for Glauber dynamics} on $G$ if there exists $\delta>0$ such that 
\begin{align*}
	\pr\brk{\min\{t:\,\sigma_t\not\in S\}\leq\emm^{\delta |V|}\mid\sigma_0\sim \mu_{G,\beta,S}}\leq\emm^{-\delta |V|}.
\end{align*}
Hence, it will most likely take Glauber an exponential amount of time to escape from a metastable state.

\begin{theorem}\label{Thm_meta}
	Let $d,q\geq3$ be integers and $\beta>0$ be real. Then, for all sufficiently small $\eps>0$, the following hold \whp{} over the choice of $\GG=\GG(n,d)$.
	\begin{enumerate}[(i)]
		\item If $\beta<\KS$, then $\Spara(\eps)$ is a metastable state for Glauber dynamics on $\GG$.
		\item If $\beta>\uniq$, then $\Sferro(\eps)$ is a metastable state for Glauber dynamics on $\GG$.
	\end{enumerate}
Further, for $\beta>\uniq$, the mixing time of Glauber  is $\emm^{\Omega(n)}$.
\end{theorem}

Thus, we can summarise the evolution of the Potts model as follows.
For $\beta<\uniq$ there is no ferromagnetic state.
As $\beta$ passes $\uniq$, the ferromagnetic state $\Sferro$ emerges first as a metastable state.
Hence, if we launch Glauber from $\Sferro$, the dynamics will most likely remain trapped in the ferromagnetic state for an exponential amount of time, even though the Boltzmann weight of the paramagnetic state is exponentially larger (as we shall see in the next section).
At the point $\curie$ the ferromagnetic state then takes over as the one dominating the Boltzmann distribution, but the paramagnetic state remains as a metastable state up to $\KS$. Note in particular that the two states coexist as metastable states throughout the interval $(\uniq,\KS)$.

The metastability for the Potts model manifests also in the evolution of the Swendsen-Wang (SW) chain, which is another popular and substantially more elaborate chain that makes non-local moves,  based on the random-cluster representation of the model. For a graph $G=(V,E)$ and a configuration $\sigma\in [q]^V$, a single iteration of SW starting from $\sigma$ consists of two steps.
\begin{itemize}
    \item \emph{Percolation step:} Let $M=M(\sigma)$ be the random edge-set obtained by adding (indepentently) each monochromatic edge under $\sigma$ with probability $p=1-\emm^{-\beta}$.
    \item \emph{Recolouring step:} Obtain the new $\sigma'\in [q]^V$ by assigning each component\footnote{Note, isolated vertices count as connected components.} of the graph $(V,M)$ a uniformly random colour from $[q]$; for $v\in V$, we set $\sigma_v'$ to be the colour assigned to $v$'s component.
\end{itemize}
We define metastable states for SW dynamics analogously to above. The following theorem establishes the analogue of Theorem~\ref{Thm_meta} for the non-local SW dynamics. Note here that SW might change the most-frequent colour due to recolouring step, so the metastability statement for the ferromagnetic phase needs to consider the set $\Sferro(\eps)$ with its $q-1$ permutations.
\begin{theorem}\label{thm:SWslow} 
	Let $d,q\geq3$ be integers and $\beta>0$ be real. Then, for all sufficiently small $\eps>0$, the following hold \whp{} over the choice of $\GG=\GG(n,d)$.
	\begin{enumerate}[(i)]
		\item If $\beta<\KS$, then $\Spara(\eps)$ is a metastable state for SW dynamics on $\GG$. 
		\item If $\beta>\uniq$, then $\Sferro(\eps)$ together with its $q-1$ permutations is a metastable state for SW dynamics on $\GG$.
	\end{enumerate}
Further, for $\beta\in(\uniq,\KS)$, the mixing time of SW is $\emm^{\Omega(n)}$.
\end{theorem}

\subsection{The relative weight of the metastable states}
At the heart of obtaining the metastability results of the previous section is a refined understanding of the relative weight of the ferromagnetic and paramagnetic states. The following notion of non-reconstruction will be the key in our arguments; it captures the absence of long-range correlations within a set $S\subset[q]^{n}$, saying that, for any vertex $v$, a typical boundary configuration on  $\vsigma_{\partial^\ell v}$ chosen according to the conditional distribution on $S$ does not impose a discernible bias on the colour of $v$ (for large $\ell,n$; recall, $\partial^\ell v$ is the set of all vertices at distance precisely $\ell$ from $v$).  More precisely, let $\mu=\mu_{\GG,\beta}$ and $\vsigma\sim \mu$; the Boltzmann distribution exhibits {\em non-reconstruction given a subset $S\subset[q]^{n}$} if for any vertex $v$ it holds that
\begin{align*}
	\lim_{\ell\to\infty}\limsup_{n\to\infty}
	\sum_{c\in[q]}\sum_{\tau\in S}\ex\brk{\mu(\tau\mid S)
	\times
	\abs{\mu(\vsigma_{v}=c\mid\vsigma_{\partial^\ell v}=\tau_{\partial^\ell v})-\mu(\vsigma_{v}=c\mid S)}}=0,
\end{align*}
where the expectation is over the choice of the graph $\GG$.

\begin{theorem}\label{Thm_Boltzmann}
	Let $d,q\geq3$ be integers and $\beta>0$ be real. The following hold for all sufficiently small $\eps>0$ as $n\to\infty$.
	\begin{enumerate}[(i)]
		\item For all $\beta<\curie$, $\ex\brk{\mu_{\GG,\beta}(\Spara)}\to1$ and, if $\beta>\uniq$, then $\ex\brk{\tfrac1n\log\mu_{\GG,\beta}(\Sferro)}\to\cB_{d,\beta}(\muferro)-\cB_{d,\beta}(\mupara)$.
		\item For all $\beta>\curie$,  $\ex\brk{\mu_{\GG,\beta}(\Sferro)}\to1/q$ and, if $\beta<\KS$, then $\ex\brk{\tfrac1n\log\mu_{\GG,\beta}(\Spara)}\to\cB_{d,\beta}(\mupara)-\cB_{d,\beta}(\muferro)$.
	\end{enumerate}
	Furthermore, the Boltzmann distribution given $\Spara$ exhibits non-reconstruction if $\beta<\KS$ and the Boltzmann distribution given $\Sferro$ exhibits non-reconstruction if $\beta>\uniq$.
\end{theorem}

\Thm~\ref{Thm_Boltzmann} shows that for $\beta<\curie$ the Boltzmann distribution is dominated by the paramagnetic state $\Spara$ for $\beta<\curie$.
Nonetheless, at $\uniq$ the ferromagnetic state $\Sferro$ and its $q-1$ mirror images start to emerge.
Their probability mass is determined by the Bethe free energy evaluated at $\muferro$.
Further, as $\beta$ passes $\curie$ the ferromagnetic state takes over as the dominant state, with the paramagnetic state lingering on as a sub-dominant state up to $\KS$.
Finally, both states $\Spara$ and $\Sferro$ are free from long-range correlations both for the regime of $\beta$ where they dominate and for those $\beta$ where they are sub-dominant.

\subsection{Discussion}\label{sec:discussion}
Our slow mixing result for Glauber dynamics when $\beta>\uniq$ (Theorem~\ref{Thm_meta}) significantly improves upon previous results of Bordewich, Greenhill and Patel~\cite{BGP} that applied to $\beta>\uniq+\Theta_q(1)$. Similarly, our slow mixing result for Swendsen-Wang dynamics when $\beta\in(\uniq,\KS)$ (Theorem~\ref{thm:SWslow}) strengthens earlier results of Galanis, \v{S}tefankovi\v{c}, Vigoda, Yang \cite{Andreas}  which applied to $\beta=\curie$, and by Helmuth, Jenseen and Perkins~\cite{HJP} which applied for  a small interval around $\curie$; both results applied only   for  $q$ sufficiently large.  To obtain our result for all integers $q,d\geq 3$, we need to carefully track how SW evolves on the random regular graph for configurations starting from the ferromagnetic and paramagnetic phases, by accounting for the percolation step via delicate arguments, whereas the approaches of \cite{Andreas,HJP} side-stepped this analysis by considering the change in the number of monochromatic edges instead.

Our slow mixing results complement the recent fast mixing result of Blanca and Gheissari \cite{BG} for edge dynamics on the random $d$-regular graph that applies to all $\beta<\uniq$. Roughly, edge dynamics is the analogue of Glauber dynamics for the random cluster representation of the Potts model (the random-cluster representation has nicer monotonicity properties). The result of \cite{BG} already implies a polynomial bound on the mixing time of SW when $\beta<\uniq$ (due to comparison results by Ullrich that apply to general graphs \cite{Ullrich}), and conversely our exponential lower bound on the mixing time of SW for $\beta\notin(\uniq,\KS)$ implies an exponential lower bound on the mixing time of edge dynamics for $\beta\notin(\uniq,\KS)$. The main open questions remaining are therefore showing whether Glauber dynamics for the Potts model mixes fast when $\beta\leq \uniq$ and whether SW/edge-dynamics mixes fast when $\beta\geq \KS$.  Extrapolating from the mean-field case (see discussion below), it is natural to conjecture that our slow mixing results are best-possible, i.e., for $\beta\leq \uniq$, Glauber mixes rapidly and similarly, for $\beta\notin(\uniq,\KS)$, SW mixes rapidly on the random regular graph.

Theorem~\ref{Thm_Boltzmann}, aside from being critical in establishing the aforementioned slow mixing and metastability results, is the first to establish for all $q,d\geq 3$ the coexistence of the ferromagnetic and paramagnetic phases for all $\beta$ in the interval  $(\uniq,\KS)$ and detail the logarithmic order of their relative weight in the same interval. Previous work in  \cite{Andreas} showed coexistence for $\beta=\curie$ (for all $q,d\geq 3$) and \cite{HJP} for $\beta$ in a small interval around $\curie$ (for large $q$ and $d\geq 5$).\footnote{We remark here that the approaches in \cite{Andreas, HJP} establish more refined estimates on the  deviations from the limiting value of the log-partition function of the phases (in the corresponding regimes they apply), with \cite{HJP} characterising in addition the limiting distribution using cluster-expansion methods. One can obtain analogous distributional characterisations for all $q,d\geq 3$ from our methods, once combined with the small subgraph conditioning method of \cite{Andreas}. It should be noted though the approach of \cite{HJP} which goes through cluster expansion is  more direct in that respect. We  don't pursue such distributional results here since  Theorem~\ref{Thm_Boltzmann} is sufficient for our slow mixing results.} Together with Theorems~\ref{Thm_meta} and ~\ref{thm:SWslow}, Theorem~\ref{Thm_Boltzmann} delineates more  firmly\footnote{Note that the interval-behaviour on the random regular graph (and hence the correspondence with the mean-field case) is already implied to some extent by the interval-result of \cite{HJP} (for $q$ large and $d\geq 5$). Note however that the interval therein is contained strictly inside $(\uniq, \KS)$ and, in particular, its endpoints do not have the probabilistic interpretation of $\uniq,\KS$. Nevertheless, \cite{HJP} obtains various probabilistic properties of the metastable phases, including a stronger form of correlation decay than that of reconstruction that we consider here.} the correspondence with the (simpler) mean-field case, the Potts model on the clique. In the mean-field case, there are qualitatively similar thresholds $\uniq,\curie,\KS$  and the mixing time for Glauber and SW have been detailed for all $\beta$, even at criticality, see \cite{BS,BSZ,GLP,GSV,CDLLPS,GJ,lee2022energy}. As mentioned earlier, the most tantalising question remaining open is to establish whether the fast mixing of SW for $\beta=\uniq$ and $\beta\geq \KS$ in the mean-field case translates to the random regular graph as well. Another interesting direction is to extend our arguments to the random-cluster representation of the Potts model for all non-integer $q\geq 1$; note that the arguments of \cite{Blanca2} and \cite{HJP} do apply to non-integer $q$ ($q\geq 1$ and $q$ large, respectively).
The proof of Theorem~\ref{Thm_Boltzmann} relies on a truncated second moment computation, an argument that was applied to different models in~\cite{Greg,ChartCoja}.

We further remark here that, from a worst-case perspective, it is known that sampling from the Potts model on $d$-regular graphs is \#BIS-hard for $\beta>\curie$ \cite{Andreas}, and we conjecture that the problem admits a poly-time approximation algorithm when $\beta<\curie$. However, even showing that Glauber mixes fast on any $d$-regular graph  in the uniqueness regime $\beta<\uniq$ is a major open problem, and Theorems~\ref{Thm_meta} and~\ref{thm:SWslow} further demonstrate that getting an algorithm all the way to $\curie$ will require using different techniques. To this end, progress has been made in \cite{UG1,UG2} where an efficient algorithm is obtained asymptotically up to $\beta_p$ for large $q$ and $d$ using cluster-expansion methods. More precise results have been shown on the random regular graph: \cite{HJP} obtained an algorithm for $d\geq 5$ and $q$ large that applies to all $\beta$ by sampling from each phase separately based also on cluster-expansion methods; also, for $\beta<\curie$, Efthymiou \cite{Charilaos} gives an algorithm with weaker approximation guarantees but which applies to all $q,d\geq 3$ (see also \cite{Blanca2}). In principle, and extrapolating again from the mean-field case, one could use Glauber/SW to sample from each phase on the random regular graph for all $q,d\geq 3$ and all $\beta$. Analysing such chains appears to be relatively far from the reach of current techniques even in the case of the random regular graph, let alone worst-case graphs. In the case of the Ising model however, the case $q=2$, the analogue of this fast mixing question has recently been established for sufficiently large $\beta$ in \cite{GS} on the random regular graph and the grid, exploiting certain monotonicity properties.

Finally, let us note that the case of the grid has qualitatively different behaviour than the mean-field and the random-regular case. There, the three critical points coincide and the behaviour at criticality depends on the value of $q$; the mixing time of Glauber and SW has largely been detailed, see \cite{BS,LS,GL}.

\section{Overview}\label{Sec_overview}

\noindent
In this section we give an overview of the proofs of \Thm s~\ref{Thm_meta}--\ref{Thm_Boltzmann}.
Fortunately, we do not need to start from first principles.
Instead, we build upon the formula for the partition function $Z_\beta(\G)$ and its proof via the second moment method from~\cite{Andreas}.
Additionally, we are going to seize upon facts about the non-reconstruction properties of the Potts model on the random $(d-1)$-ary tree, also from~\cite{Andreas}.
We will combine these tools with an auxiliary random graph model known as the planted model, which also plays a key role in the context of inference problems on random graphs~\cite{CKPZ}.

\subsection{Preliminaries}\label{Sec_pre}
Throughout most of the paper, instead of the simple random regular graph $\GG$, we are going to work with the random $d$-regular multi-graph $\G=\G(n,d)$ drawn from the pairing model.
Recall that $\G$ is obtained by creating $d$ clones of each of the vertices from $[n]$, choosing a random perfect matching of the complete graph on $[n]\times[d]$ and subsequently contracting the vertices $\{i\}\times[d]$ into a single vertex $i$, for all $i\in [n]$.
It is well-known that $\GG$ is contiguous with respect to $\G$~\cite{JLR}, i.e., any property that holds \whp{} for $\G$ also holds \whp{} for $\GG$. 


For a configuration $\sigma\in[q]^{V(G)}$ define a probability distribution $\nu^\sigma$ on $[q]$ by letting
\begin{align*}
	\nu^\sigma(s)&=|\sigma^{-1}(s)|/n&&(s\in[q]).
\end{align*}
In words, $\nu^\sigma$ is the empirical distribution of the colours under $\sigma$.
Similarly, let $\rho^{G,\sigma}\in\cP([q]\times[q])$ be the edge statistics of a given graph/colouring pair, i.e.,
\begin{align*}
	\rho^{G,\sigma}(s,t)&=\frac1{2|E(G)|}\sum_{u,v\in V(G)}\vecone\{uv\in E(G),\,\sigma_u=s,\,\sigma_v=t\}.
\end{align*}
We are going to need the following elementary estimate of the number of $d$-regular multigraphs $G$ that attain a specific $\rho^{G,\sigma}$.

\begin{lemma}[{\cite[\Lem~2.7]{Samuel}}]\label{Lemma_Samuel}
	Suppose that $\sigma\in[q]^n$. 	Moreover, suppose that $\rho=(\rho(s,t))_{s,t\in[q]}$ is a symmetric matrix with positive entries such that $dn\rho(s,t)$ is an integer for all $s,t\in[q]$, $dn\rho(s,s)$ is even for all $s\in[q]$ and
	$\sum_{t=1}^q\rho(s,t)=\nu^{\sigma}(s)$ for all $s\in[q]$. 
	Let $\cG(\sigma,\rho)$ be the event that $\rho^{\G,\sigma}=\rho$. 
	Then
	\begin{align*}
		\pr\brk{\cG(\sigma,\rho)}&=\exp\brk{\frac{dn}2\sum_{s,t=1}^q\rho(s,t)\log\frac{\nu^{\sigma}(s)\nu^{\sigma}(t)}{\rho(s,t)}+O(\log n/n)}.
	\end{align*}
\end{lemma}

\subsection{Moments and messages}\label{Sec_mm}
The routine method for investigating the partition function and the Boltzmann distribution of random graphs is the method of moments~\cite{ANP}.
The basic strategy is to calculate, one way or another, the first two moments $\ex[Z_\beta(\G)]$, $\ex[Z_\beta(\G)^2]$ of the partition function.
Then we cross our fingers that the second moment is not much larger than the square of the first. 
It sometimes works.
But potential pitfalls include a pronounced tendency of running into extremely challenging optimisation problems in the course of the second moment calculation and, worse, lottery effects that may foil the strategy altogether.
While regular graphs in general and the Potts ferromagnet in particular are relatively tame specimens, these difficulties actually do arise once we set out to investigate metastable states.
Drawing upon~\cite{Victor,Greg} to sidestep these challenges, we develop a less computation-heavy proof strategy.

The starting point is the observation that the fixed points of~\eqref{eqBP} are intimately related to the moment calculation.
This will not come as a surprise to experts, and indeed it was already noticed in~\cite{Andreas}. 
To elaborate, let $\nu=(\nu(s))_{s\in[q]}$ be a probability distribution on the $q$ colours.
Moreover, let $\cR(\nu)$ be the set of all symmetric matrices $\rho=(\rho(s,t))_{s,t\in[q]}$ with non-negative entries such that
\begin{align}
	\sum_{t\in[q]}\rho(s,t)&=\nu(s)\qquad \mbox{for all }s\in[q]. \label{eqRhoTau}
\end{align}
Simple manipulations (e.g., \cite[\Lem~2.7]{Samuel}) show that the first moment satisfies
\begin{equation}\label{eqFirstMmtBethe}
\begin{aligned}
	\lim_{n\to\infty}\frac{1}{n}\log\ex[Z_\beta(\G)]&=\max_{\nu\in \cP([q]), \rho\in R(\nu)}F_{d,\beta}(\nu,\rho),\qquad\mbox{where}\\
	F_{d,\beta}(\nu,\rho)&=(d-1)\sum_{s\in[q]}\nu(s)\log\nu(s)-d\sum_{1\leq s\leq t \leq q}\rho(s,t)\log\rho(s,t)+\frac{d\beta}{2}\sum_{s\in[q]}\rho(s,s).
	\end{aligned}
\end{equation}
Thus, the first moment is governed by the maximum or maxima, as the case may be, of $F_{d,\beta}$. 

We need to know that the maxima of $F_{d,b}$ are in one-to-one correspondence with the stable fixed points of~\eqref{eqBP}.
To be precise, a fixed point $\mu$ of \eqref{eqBP} is \emph{stable} if the Jacobian of \eqref{eqBP} at $\mu$ has spectral radius strictly less than one.
Let $\Fplus$ be the set of all stable fixed points $\mu\in\Fix$.
Moreover, let $\Fone$ be the set of all $\mu\in\Fplus$ such that $\mu(1)=\max_{s\in[q]}\mu(s)$.
In addition, let us call a local maximum $(\nu,\rho)$ of $F_{d,\beta}$ \emph{stable} if there exist $\delta,c>0$ such that 
\begin{align}\label{eqstabmax}
	F_{d,\beta}(\nu',\rho')&\leq F_{d,\beta}(\nu,\rho)-c\bc{\|\nu-\nu'\|^2+\|\rho-\rho'\|^2}
\end{align}
for all $\nu'\in \cP([q])$ and $\rho'\in \cR(\nu')$ such that $\|\nu-\nu'\|+\|\rho-\rho'\|<\delta$.
Roughly, \eqref{eqstabmax} provides that the Hessian of $F_{d,\beta}$ is negative definite on the subspace of all possible $\nu,\rho$.

\begin{lemma}[{\cite[\Thm~8]{Andreas}}]\label{Lemma_Jean}
	Suppose that $d\ge3,\beta>0$.
	The map $\mu\in\cP([q])\mapsto(\nu^\mu,\rho^\mu)$ defined by
	\begin{align}\label{eqLemma_Jean}
		\nu^\mu(s)&=\frac{(1+(\eul^\beta-1)\mu(s))^d}{\sum_{t\in[q]}(1+(\eul^\beta-1)\mu(t))^d},&\rho^\mu(s,t)&=\frac{\eul^{\beta\vecone\{s=t\}}\mu(s)\mu(t)}{1+(\eul^\beta-1)\sum_{s\in[q]}\mu(s)^2}
	\end{align}
	is a bijection from $\Fplus$ to the set of stable local maxima of $F_{d,\beta}$.
	Moreover, for any  fixed point $\mu$ we have $\cB_{d,\beta}(\mu)=F_{d,\beta}(\nu^\mu,\rho^\mu).$
\end{lemma}

For brevity, let $(\nupara,\rhopara)=(\nu^{\mupara},\rho^{\mupara})$ and $(\nuferro,\rhoferro)=(\nu^{\muferro},\rho^{\muferro})$. The following result characterises the stable fixed points $\Fone$.

\begin{proposition}[{\cite[Theorem 4]{Andreas}}]\label{Prop_max}
Suppose that $d\ge3,\beta>0$.
\begin{enumerate}[(i)]
	\item If $\beta<\uniq$, then \eqref{eqBP} has a unique fixed point, namely the paramagnetic distribution $\nupara$ on $[q]$.
		This fixed point is stable and thus $F_{d,\beta}$ attains its global maximum at $(\nupara,\rhopara)$.
	\item If $\uniq<\beta<\KS$, then $\Fone$ contains two elements, namely the paramagnetic distribution $\nupara$ and  the   ferromagnetic distribution $\nuferro$; $(\nupara,\rhopara)$ is a global maximum of $F_{d,\beta}$ iff $\beta\leq \curie$, and  $(\nuferro,\rhoferro)$ iff $\beta\geq \curie$.
	\item If $\beta>\KS$, then $\Fone$ contains precisely one element, namely the ferromagnetic distribution $\nuferro$, and $(\nuferro,\rhoferro)$ is a global maximum  of $F_{d,\beta}$.
\end{enumerate}
\end{proposition}

Like the first moment, the second moment boils down to an optimisation problem as well, albeit one of much higher dimension ($q^2-1$ rather than $q-1$).
Indeed, it is not difficult to derive the following approximation (once again, e.g., via \cite[\Lem~2.7]{Samuel}).
For a probability distribution $\nu\in\cP([q])$ and a symmetric matrix $\rho\in\cR(\nu)$ let $\cR^\tensor(\rho)$ be the set of all tensors $r=(r(s,s',t,t'))_{s,s',t,t'\in[q]}$ such that
\begin{align} \label{eqFdB2}
r(s,s',t,t')&=r(t,t',s,s')&&\mbox{and}&&\sum_{s',t'\in[q]}r(s,s',t,t')=\sum_{s',t'\in[q]}r(s',s,t',t)=\rho(s,t)\quad\mbox{for all }s,t\in[q].
\end{align}
Then, with $H(\cdot)$ denoting the entropy function, we have
\begin{align}\nonumber
	\lim_{n\to\infty}\frac{1}{n}\log\ex[(Z_{\beta}(\G))^2]&=\max_{\nu,\rho\in\cR(\nu),r\in\cR^\tensor(\rho)}F_{d,\beta}^\tensor(\rho,r),\mbox{ where }\\
F_{d,\beta}^\tensor(\rho,r)&=(d-1)H(\rho)+\frac{d}{2}H(r)+\frac{d\beta}{2}\sum_{s,s',t,t'\in[q]}  \bc{\vecone\{s=t\}+\vecone\{s'=t'\}} r(s, s', t, t') .
\label{eqsmmF}
\end{align}
A frontal assault on this optimisation problem is in general a daunting task due to the doubly-stochastic constraints in \eqref{eqFdB2}, i.e., the constraint $r\in\cR^\tensor(\rho)$.
But fortunately, to analyse the global maximum (over $\nu$ and $\rho$), these constraints can be relaxed, permitting an elegant translation of the problem to operator theory.
In effect, the second moment computation can be reduced to a study of matrix norms.
The result can be summarised as follows.

\begin{proposition}[{\cite[Theorem 7]{Andreas}}]\label{Prop_Andreas}
	For all $d,q\geq3$ and $\beta>0$ we have $\max_{\nu,\rho\in\cR(\nu),r\in\cR^\tensor(\rho)}F_{d,\beta}^\tensor(\rho,r)=2\max_{\nu,\rho}F_{d,\beta}(\nu,\rho)$ and thus $\ex[Z_\beta(\G)^2]=O(\ex[Z_\beta(\G)]^2)$.
\end{proposition}

\noindent
Combining \Lem~\ref{Lemma_Jean}, \Prop~\ref{Prop_max} and \Prop~\ref{Prop_Andreas}, we obtain the following reformulation of \cite[Theorem 7]{Andreas}, which  verifies that we obtain good approximations to the partition function by maximising the Bethe free energy on $\Fix$.
\begin{theorem}\label{Thm_Z}
	For all integers $d,q\geq3$ and real $\beta>0$, we have $\displaystyle\lim_{n\to\infty}n^{-1}\log Z_{\beta}(\GG)=\max_{\mu\in\Fix}\cB_{d,\beta}(\mu)$ in probability.
\end{theorem}

\subsection{Non-reconstruction}\label{Sec_nr}
While the {\em global} maximisation of the function $F_{d,\beta}^\tensor$ and thus the proof of \Thm~\ref{Thm_Z} boils down to matrix norm analysis, in order to prove \Thm s~\ref{Thm_Boltzmann} and~\ref{Thm_meta} via the method of moments we would in addition need to get a good handle on all the {\em local} maxima.
Unfortunately, we do not see a way to reduce this more refined question to operator norms.
Hence, it would seem that we should have to perform a fine-grained analysis of $F_{d,\beta}^\tensor$ after all.
But luckily another path is open to us.
Instead of proceeding analytically, we resort to probabilistic ideas.
Specifically, we harness the notion of non-reconstruction on the Potts model on the $d$-regular tree.

To elaborate, let $\TT_{d}$ be the infinite $d$-regular tree with root $o$.
Given a probability distribution $\mu\in\{\mupara,\muferro\}$ we define a broadcasting process $\vsigma=\vsigma_{d,\beta,\mu}$ on $\TT_{d}$ as follows.
Initially we draw the colour $\vsigma_{o}$ of the root $o$ from the distribution $\nu^{\mu}$.
Subsequently, working our way down the levels of the tree, the colour of a vertex $v$ whose parent $u$ has been coloured already is drawn from the distribution
\begin{align*}
\pr\brk{\vsigma_{v}=c\mid\vsigma_{u}}&=\frac{\mu(c)\eul^{\beta\vecone\{c=\vsigma_{u}\}}}{\sum_{c'\in[q]}\mu(c')\eul^{\beta\vecone\{c'=\vsigma_{u}\}}}.
\end{align*}
Naturally, the colours of different vertices on the same level are mutually independent conditioned on the parent's colour.
Let $\partial^\ell o$ be the set of all vertices at distance precisely $\ell$ from $o$.
We say that the broadcasting process has the {\em strong non-reconstruction property} if 
\begin{align*}
	\sum_{c\in[q]} \ex\Big[\big|\pr\brk{\vsigma_{o}=c\mid\vsigma_{\partial^\ell o}}-\pr\brk{\vsigma_{o}=c}\big|\Big]=\exp(-\Omega(\ell)),
\end{align*}
where the expectation is over the random configuration $\vsigma_{\partial^\ell o}$ (distributed according to the broadcasting process). In words, this says that the information of the spin of the root decays in the broadcasting process; the term ``strong'' refers that the decay is exponential with respect to the depth $\ell$.

\begin{proposition}[{\cite[Theorem 50]{Andreas}}]\label{Prop_nonre}
Let $d,q\geq3$ be integers and $\beta>0$ be real.
\begin{enumerate}[(i)]
	\item For $\beta<\KS$, the broadcasting process $\vec\sigma_{d,\beta,\mupara}$ has the strong non-reconstruction property.
	\item For $\beta>\uniq$, the broadcasting process $\vec\sigma_{d,\beta,\muferro}$ has the strong non-reconstruction property.
\end{enumerate}
\end{proposition}

In order to prove \Thm s~\ref{Thm_meta} -- \ref{Thm_Boltzmann} we will combine \Prop~\ref{Prop_nonre} with reweighted random graph models known as planted models.
To be precise, we will consider two versions of planted models, a paramagnetic and a ferromagnetic one.
Then we will deduce from \Prop~\ref{Prop_nonre} that the Boltzmann distribution of these planted models has the non-reconstruction property in a suitably defined sense.
In combination with some general facts about Boltzmann distributions this will enable us to prove \Thm s~\ref{Thm_meta} -- \ref{Thm_Boltzmann} without the need for complicated moment computations.

\subsection{Planting}\label{Sec_planting}
We proceed to introduce the paramagnetic and the ferromagnetic version of the planted model.
Roughly speaking, these are weighted versions of the common random regular graph $\GG$ where the probability mass of a specific graph is proportional to the paramagnetic or ferromagnetic bit of the partition function.
To be precise, for $\eps>0$, recall the subsets $\Spara=\Spara(\eps),\Sferro=\Sferro(\eps)$ of the configuration space $[q]^n$.
Letting
\begin{align}\label{eqZfp}
	\Zferro(G)=\sum_{\sigma\in\Sferro}\emm^{\beta\cH_{G}(\sigma)}\mbox{ and }\Zpara(G)=\sum_{\sigma\in\Spara}\emm^{\beta\cH_{\G}(\sigma)},
\end{align}
we define random graph models $\Gferro,\Gpara$ by
\begin{align}\label{eqDefPlanted}
	\pr\brk{\Gferro=G}&=\frac{\Zferro(G)\pr\brk{\G=G}}{\Erw[\Zferro(\G)]},&
	\pr\brk{\Gpara=G}&=\frac{\Zpara(G)\pr\brk{\G=G}}{\Erw[\Zpara(\G)]}.
\end{align}
Thus, $\Gferro$ and $\Gpara$ are $d$-regular random graphs on $n$ vertices such that the probability that a specific graph $G$ comes up is proportional to $\Zferro(G)$ and $\Zpara(G)$, respectively.

We need to extend the notion of non-reconstruction to $\Gpara,\Gferro$.
Specifically, we need to define non-reconstruction for the conditional Boltzmann distributions $\mu_{\GG,\beta}(\nix\mid\Spara)$, $\mu_{\GG,\beta}(\nix\mid\Sferro)$.
We thus say that for a graph/configuration pair $(G,\sigma)$, an event $S\subset[q]^{n}$, a positive real $\xi>0$, a real number $\gamma \in [0,1]$, an integer $\ell\geq1$ and a probability distribution $\mu$ on $[q]$ the {\em conditional $(\gamma, \xi,\ell,\mu)$-non-reconstruction property} holds if
\begin{align}\label{eqNonRec}
\frac1n\sum_{v\in [n]}\sum_{c\in[q]}\abs{\nu^\mu(c)-\mu_{G,\beta}(\vsigma_{G,\beta,v}=c\mid S,\vsigma_{G,\beta,\partial^\ell v}=\sigma_{\partial^\ell v})}&<\xi
\end{align}
holds with probability $1-\gamma$. In words, \eqref{eqNonRec} provides that on the average over all $v$ the conditional marginal probability $\mu_{G,\beta}(\vsigma_{G,\beta,v}=c\mid S,\vsigma_{G,\beta,\partial^\ell v}=\sigma_{\partial^\ell v})$ that $v$ receives colour $c$ given the boundary condition induced by $\sigma$ on the vertices at distance $\ell$ from $v$ and given the event $S$ is close to $\nu^\mu(\omega)$.

Further, while \eqref{eqNonRec}
deals with a specific graph/configuration pair $(G,\sigma)$, we need to extend the definition to the random graph models $\Gferro$ and $\Gpara$.
For a graph $G$ let $\vsigma_{G,\mathrm f}$ denote a sample from the conditional distribution $\mu_{G,\beta}(\nix\mid\Sferro)$.
Also define $\vsigma_{G,\mathrm p}$ similarly for $\Spara$.
We say that for the random graph $\Gferro$ has the {\em $(\eta,\xi,\ell)$-non-reconstruction property} if
\begin{align}\label{eqNonRecGferro}
	\ex\brk{\mu_{\Gferro,\beta}\bc{\cbc{\mbox{$(\Gferro,\vsigma_{\Gferro,\mathrm{f}})$ fails to have the $(\xi,\ell,\muferro)$-non-reconstruction property conditional on $\Sferro$}}}}<\eta.
\end{align}
Thus, we ask that \eqref{eqNonRec}
holds for a typical graph/configuration pair obtained by first drawing a graph $\Gferro$ from the planted model and then sampling $\vsigma_{\Gferro,\mathrm f}$ from $\mu_{\Gferro}(\nix\mid\Sferro)$.
We introduce a similar definition for $\Gpara$.

The following proposition shows that the non-reconstruction statements from \Prop~\ref{Prop_nonre} carry over to the planted random graphs.
This is the key technical statement toward the proofs of \Thm s~\ref{Thm_meta}--\ref{Thm_Boltzmann}.

\begin{proposition} \label{Prop_nonreG}
Let $d\geq3$.
\begin{enumerate}[(i)]
	\item Assume that $\uniq <\beta$.
		Then $\Gferro$ has the $(o(1),1/\log\log n,\lceil \log\log n\rceil)$-non-reconstruction property.
		Moreover, for any $\delta>0$ there exist $\ell=\ell(d,\beta,\delta)>0$ and $\chi=\chi(d,\beta,\delta)>0$ such that $(\Gferro,\vsigma_{\Gferro,\mathrm{f}})$ has the $(\exp(-\chi n),\delta,\ell,\muferro)$-non-reconstruction property.
	\item Assume that $\beta<\KS$.
		Then $\Gpara$ has the $(o(1),1/\log\log n,\lceil \log\log n\rceil)$-non-reconstruction property.
		Moreover, for any $\delta>0$ there exist $\ell=\ell(d,\beta,\delta)>0$ and $\chi=\chi(d,\beta,\delta)>0$ such that $(\Gpara,\vsigma_{\Gpara,\mathrm{p}})$ has the $(\exp(-\chi n),\delta,\ell,\mupara)$-non-reconstruction property.
\end{enumerate}
\end{proposition}

Together with a few routine arguments for the study of Boltzmann distributions that build upon~\cite{Victor}, we can derive from \Prop~\ref{Prop_nonreG} that for $\beta>\uniq$ two typical samples from the ferromagnetic Boltzmann distribution have overlap about $\nuferro\tensor\nuferro$.
This insight enables a truncated second moment computation that sidesteps a detailed study of the function $F_{d,\beta}^\tensor$ from \eqref{eqsmmF}.
Indeed, the only observation about $F^\tensor_{d,\beta}$ that we need to make is that $F^\tensor_{d,\beta}(\nuferro\tensor\nuferro,\rhoferro\tensor\rhoferro)=2F_{d,\beta}(\nuferro,\rhoferro)$.
Similar arguments apply to the paramagnetic case.
We can thus determine the asymptotic Boltzmann weights of $\Spara,\Sferro$ on the random regular graph as follows.

\begin{corollary}\label{Cor_nonreG}
	Let $d,q\geq3$ be arbitrary integers.
	\begin{enumerate}[(i)]
	\item For  $\beta>\uniq$, for all sufficiently small $\eps>0$,  we have \whp\ $\frac1n\log\Zferro(\G)=\cB_{d,\beta}(\muferro)+o(1)$.
		\item For $\beta<\KS$, for all sufficiently small $\eps>0$,  we have \whp\ $\frac1n\log\Zpara(\G)=\cB_{d,\beta}(\mupara)+o(1)$.
	\end{enumerate}
\end{corollary}

Finally, combining \Cor~\ref{Cor_nonreG} with the definition \eqref{eqDefPlanted} of the planted models and the non-reconstruction statements from \Prop~\ref{Prop_nonreG}, we obtain the following conditional non-reconstruction statements for the plain random regular graph.

\begin{corollary}\label{Cor_cutm}
Let $d,q\geq3$ be arbitrary integers.
\begin{enumerate}[(i)]
	\item For  $\beta>\uniq$, the Boltzmann distribution $\mu_{\G,\beta}$ given $\Sferro$ exhibits the non-reconstruction property.
	\item For $\beta<\KS$, the Boltzmann distribution $\mu_{\G,\beta}$ given $\Spara$ exhibits the non-reconstruction property.
\end{enumerate}
\end{corollary}

\noindent
\Thm~\ref{Thm_Boltzmann} is an immediate consequence of Corollaries~\ref{Cor_nonreG} and~\ref{Cor_cutm}.

\section{Quiet planting}\label{sec:quiet}

\noindent
In this section we prove \Prop~\ref{Prop_nonreG} along with Corollaries~\ref{Cor_nonreG} and~\ref{Cor_cutm}.
We begin with an important general observation about the planted model called the Nishimori identity, which will provide an explicit constructive description of the planted models.

\subsection{The Nishimori identity}\label{Sec_Nishi}
We complement the definition \eqref{eqDefPlanted} of the planted random graphs $\Gferro,\Gpara$ by also introducing a reweighted distribution on graphs for a specific configuration $\sigma\in[q]^{n}$.
Specifically, we define a random graph $\hat\G(\sigma)$ by letting
\begin{align}\label{eqGsigma}
	\pr\brk{\hat\G(\sigma)=G}&=\frac{\pr\brk{\G=G}\emm^{\beta\cH_{\G}(\sigma)}}{\Erw[\emm^{\beta\cH_{\G}(\sigma)}]}.
\end{align}
Furthermore, for $\eps>0$, recalling the truncated partition functions $\Zferro,\Zpara$ from \eqref{eqZfp}, we introduce reweighted random configurations $\sferro=\sferro(\eps)\in [q]^n$ and $\spara=\spara(\eps)\in[q]^{n}$ with distributions
\begin{align}\label{eqsferrospara}
	\pr\brk{\sferro=\sigma}&=\frac{\vecone\cbc{\sigma\in\Sferro}\Erw[\emm^{\beta\cH_{\G}(\sigma)}]}{\Erw[\Zferro(\G)]},&
	\pr\brk{\spara=\sigma}&=\frac{\vecone\cbc{\sigma\in\Spara}\Erw[\emm^{\beta\cH_{\G}(\sigma)}]}{\Erw[\Zpara(\G)]}.
\end{align}
We have the following paramagnetic and ferromagnetic {\em Nishimori identities}.

\begin{proposition}\label{Prop_Nishi}
	For any integers $d,q\geq3$ and real $\beta,\eps>0$,  we have 
	\begin{align}\label{eqProp_Nishi}
		(\Gpara,\sparaGpara)\disteq(\hat\G(\spara),\spara),&&
		(\Gferro,\sferroGferro)\disteq(\hat\G(\sferro),\sferro).
	\end{align}
\end{proposition}
\begin{proof} 
	Let $G$ be a $d$-regular graph on $n$ vertices and $\sigma \in [q]^n$. We have
	\begin{align}
		\pr\brk{   (\Gpara, \sparaGpara)=\bc{G,\sigma  }} &= \pr\brk{\sparaGpara = \sigma  \Big\vert \Gpara = G} \pr\brk{ \Gpara=G}  = \mu_{G,\beta} \bc{\sigma \Big\vert \Spara } \frac{\Zpara(G)\pr\brk{\G=G}}{\Erw[\Zpara(\G)]} \label{eq:PGparaSpara}.
	\end{align}
	Moreover, by the definition of the Boltzmann distribution $\mu_{G,\beta}$,
	\begin{align}
		\mu_{G,\beta} \bc{\sigma \vert \Spara } &= \frac{\emm^{\beta\cH_{\G}(\sigma)}\vecone \cbc{ \sigma \in \Spara}  }{ Z(G) \ \mu_{G,\beta}\bc{\Spara}},& 
		\mu_{G}\bc{\Spara} &= \frac{ \Zpara(G) }{ Z(G)}\label{eq:muSigmaSpara}.
	\end{align}
	Combining \eqref{eq:PGparaSpara} and \eqref{eq:muSigmaSpara}, we obtain
	\begin{align*}
		\pr\brk{  (\Gpara, \sparaGpara) =(G, \sigma) } &= \frac{\emm^{\beta\cH_{\G}(\sigma)} \pr\brk {\G=G} }{ \Erw\brk{ \emm^{\beta\cH_{\G}(\sigma)}}} \cdot \frac{\Erw\brk{\emm^{\beta\cH_{\G}(\sigma)}}\vecone \cbc{ \sigma \in \Spara}  }{ \Erw[ \Zpara(\G)]} \\
		  &= \pr\brk{ \hat \G\bc{ \spara } = G  \Big\vert \spara=\sigma } \pr\brk{ \spara = \sigma}=\pr\brk{(\hat\G(\spara)=G,\spara=\sigma)},
	\end{align*}
as claimed.
The very same steps apply to $\Gferro$.
\end{proof}

Nishimori identities were derived in~\cite{CKPZ} for a broad family of planted models which, however, does not include the planted ferromagnetic models $\Gpara,\Gferro$.
Nonetheless, the (simple) proof of \Prop~\ref{Prop_Nishi} is practically identical to the argument from~\cite{CKPZ}.

While the original definition \eqref{eqDefPlanted} of the planted models may appear unwieldy, \Prop~\ref{Prop_Nishi} paves the way for a more hands-on description.
But as a preliminary step we need to get a handle on the empirical distribution of the colours under the random configurations $\sferro,\spara$.
Additionally, we also need to determine the edge statistics $\rho^{\Gpara,\spara}$ and $\rho^{\Gferro,\sferro}$.
The following two lemmas solve these problems for us.

\begin{lemma}\label{lemma_EZp}
	Suppose that $0\leq\beta<\beta_h$.
	Then $\Erw[\Zpara(\G)]=\exp(nF_{d,\beta}(\nupara,\rhopara)+O(\log n))$.
\end{lemma}
\begin{proof}
	To obtain a lower bound on $\Erw[\Zpara(\G)]$ let $\sigma_0\in[q]^n$ be a configuration such that $|\sigma_0^{-1}(s)|=\frac nq+O(1)$ for all $s\in[q]$.
	Let $\nu(s)=|\sigma_0^{-1}(s)|/n$.
	Then
		\begin{align*}
			\rho^\nu(s,t)&=\frac{\eul^{\beta\vecone\{s=t\}}}{q(q-1+\eul^\beta)}+O(1/n)=\rhopara(s,t)+O(1/n)&&(s,t\in[q]).
		\end{align*}
	Therefore, \Lem~\ref{Lemma_Samuel} yields
	\begin{align}\nonumber
		\Erw[\Zpara(\G)]&\geq\sum_{\sigma\in[q]^n}\vecone\cbc{\forall s\in[q]:|\sigma^{-1}(s)|=n\nu(s)}\pr\brk{\cG(\sigma,\rho^\nu)}\exp\bc{\frac{\beta \eul^\beta dn}{2(q-1+\eul^\beta)}+O(1)}\\
						&\geq q^n\exp\bc{\frac{\beta \eul^\beta dn}{2(q-1+\eul^\beta)}+O(\log n)}=\exp\bc{nF_{d,\beta}(\nupara,\rhopara)+O(\log n)}.\label{eq_lemma_EZp_1}
		\end{align}
	Conversely, since there are only $n^{O(1)}$ choices of $\nu,\rho$, \Lem~\ref{Lemma_Jean} and \Prop~\ref{Prop_max} imply that
		\begin{align}\label{eq_lemma_EZp_2}
		\Erw[\Zpara(\G)]&\leq \exp\bc{nF_{d,\beta}(\nupara,\rhopara)+O(\log n)}.
		\end{align}
	The assertion follows from \eqref{eq_lemma_EZp_1} and \eqref{eq_lemma_EZp_2}.
\end{proof}

\begin{lemma}\label{lemma_EZf}
	Suppose that $\beta>\beta_u$.
	Then $\Erw[\Zferro(\G)]=\exp(nF_{d,\beta}(\nuferro,\rhoferro)+O(\log n))$.
\end{lemma}
\begin{proof}
	As in the proof of \Lem~\ref{lemma_EZp} let $\sigma_0\in[q]^n$ be a configuration such that $|\sigma_0^{-1}(s)|=n\nuferro(s)+O(1)$ for all $s\in[q]$.
	Letting $\nu(s)=|\sigma_0^{-1}(s)|/n$ we see that $\rho^{\nu}(s,t)=\rhoferro(s,t)+O(1/n)$ for all $s,t\in[q]$.
	Therefore, \Lem~\ref{Lemma_Samuel} yields
	\begin{align}\nonumber
		\Erw[\Zferro(\G)]&
						\geq \binom{n}{\nu n}\exp\bc{-\frac{dn}2\KL{\rho^\nu}{\nu\tensor\nu}
	+\frac{\beta\eul^{\beta} dn\sum_{s\in[q]}\muferro(s)^2}{2\big(1+(\eul^\beta-1)\sum_{s\in[q]}\muferro(s)^2\big)}+O(\log n)}\\
							&=\exp\bc{nF_{d,\beta}(\nuferro,\rhoferro)+O(\log n)}.\label{eq_lemma_EZf_1}
	\end{align}
	As for the upper bound, once again because there are only $n^{O(1)}$ choices of $\nu,\rho$, \Lem~\ref{Lemma_Jean} and \Prop~\ref{Prop_max} yield
		\begin{align}\label{eq_lemma_EZf_2}
		\Erw[\Zferro(\G)]&\leq \exp\bc{nF_{d,\beta}(\nuferro,\rhoferro)+O(\log n)}.
		\end{align}
	Combining the lower and upper bounds from \eqref{eq_lemma_EZf_1} and \eqref{eq_lemma_EZf_2} completes the proof.
\end{proof}

\begin{lemma}\label{Lemma_spara}
For any integers $d,q\geq3$ and real $\beta\in(0,\beta_h)$, there exist $c,t_0>0$ such that 
\begin{align*}
\pr\brk{\dTV\bc{\nu^{\spara},\nupara }+\dTV\bc{\rho^{\hG(\spara),\spara},\rhopara}>t}\leq\exp(-ct^2n+O(\log n))\qquad\mbox{for all }0\leq t<t_0.
\end{align*}
\end{lemma}
\begin{proof}
	Suppose that $\nu$ is a probability distribution on $[q]$ such that $n\nu(s)$ is an integer for all $s\in[q]$.
	Moreover, suppose that $\rho=(\rho(s,t))_{s,t\in[q]}$ is a symmetric matrix such that $dn\rho(s,t)$ is an integer for all $s,t\in[q]$, $dn\rho(s,s)$ is even for all $s\in[q]$ and
	$\sum_{t=1}^q\rho(s,t)=\nu(s)$ for all $s\in[q]$.
	Retracing the steps of the proof of \Lem~\ref{lemma_EZf}, we see that
	\begin{align}\label{eq:summingsigma}
		\sum_{\sigma\in[q]^n}\vecone\cbc{\nu^{\sigma}=\nu}\pr\brk{\cG(\sigma,\rho)}\exp\bc{\frac{\beta dn}2\sum_{s=1}^q\rho(s,s)}
						&=\exp\bc{nF_{d,\beta}(\nu,\rho)+O(\log n)}.
		\end{align}
Therefore, the assertion follows from \Prop~\ref{Prop_max} and the definition \eqref{eqstabmax} of stable local maxima. 
\end{proof}

\begin{lemma}\label{Lemma_sferro}
For any integers $d,q\geq3$ and real $\beta>\uniq$, there exist $c,t_0>0$ such that
\begin{align*}
	\pr\brk{\dTV\bc{\nu^{\sferro},\nuferro}+\dTV\bc{\rho^{\hG(\sferro),\sferro},\rhoferro}>t}\leq\exp(-ct^2n+O(\log n))\qquad\mbox{for all }0\leq t<t_0.
\end{align*}
\end{lemma}
\begin{proof}
	The argument from the proof of \Lem~\ref{Lemma_spara} applies {\em mutatis mutandis}.
\end{proof}

At this point we have handy, constructive descriptions of the models $\Gpara,\Gferro$.
Indeed, \Lem s~\ref{Lemma_spara} and~\ref{Lemma_sferro} provide that the planted configurations $\spara$ and $\sferro$ have colour statistics approximately equal to $\nupara$ and $\nuferro$ \whp, respectively.
Moreover, because the random graph models are invariant under permutations of the vertices, $\spara$ and $\sferro$ are uniformly random given their colour statistics.
In addition, the edge statistics of the random graphs $\hG(\spara)$ and $\hG(\sferro)$ concentrate about $\rhoferro$ and $\rhopara$.
Once more because of permutation invariance, the random graphs $\hG(\spara)$ and $\hG(\sferro)$ themselves are uniformly random given the planted assignment $\spara$ or $\sferro$ and given the edge statistics.

Thus, let $\fSferro$ and $\fSpara$ be the $\sigma$-algebras generated by $\sferro,\rho^{\Gferro,\sferro}$ and $\spara,\rho^{\Gpara,\spara}$, respectively.
Then we can use standard techniques from the theory of random graphs to derive typical properties of $\hG(\spara)$ given $\fSpara$ and of $\hG(\sferro)$ given $\fSferro$, which are distributed precisely as $\Gpara$ and $\Gferro$ by \Prop~\ref{Prop_Nishi}.
Using these characterisations, we are now going to prove \Prop~\ref{Prop_nonreG}.

\subsection{Proof of \Prop~\ref{Prop_nonreG}}\label{Sec_Prop_nonreG}
\Lem~\ref{Lemma_spara} gives sufficiently accurate information as to the distribution of $\spara,\rho^{\Gpara,\spara}$ for us to couple the distribution of the colouring produced by the broadcasting process and the colouring that $\spara$ induces on the neighbourhood of some particular vertex of $\Gpara$, say $v$.

\begin{lemma}\label{Lemma_nonrecPara}
Let $d,q\geq3$ be integers and $\beta\in(0,\KS)$ be real.
Then, for any vertex $v$ and any non-negative integer $\ell=o(\log n)$, given $\fSpara$ \whp\ we have 
\begin{align*}
\dTV(\hat\SIGMA_{\mathrm{p},\partial^\ell v},\tau_{\partial^\ell o})=O\bc{d^\ell\bc{\dTV(\nu^{\spara},\nupara)+\dTV(\rho^{\hG(\spara),\spara},\rhopara)+O(n^{-0.99})}}.
\end{align*}
\end{lemma}
\begin{proof}
	Proceeding by induction on $\ell$, we construct a coupling of $\hat\SIGMA_{\mathrm{p},\partial^\ell v}$ and $\tau_{\partial^\ell o}$.
	Let
	\begin{align}\label{eqLemma_nonrecPara1}
		\zeta=\dTV(\nu^{\spara},\nupara)+\dTV(\rho^{\hG(\spara),\spara},\rhopara).
	\end{align}
	In the case $\ell=0$ the set $\partial^\ell v$ consists of $v$ only, while $\partial^\ell o$ comprises only the root vertex $o$ itself.
	Hence, the colours $\spara(v)$ and $\tau(o)$ can be coupled to coincide with probability at least $1-\zeta$.
	As for $\ell\geq1$, assume by induction that $\partial^{\ell-1} v$ is acyclic and that $\hat\SIGMA_{\mathrm{p},\partial^{\ell-1} v}$ and $\tau_{\partial^{\ell-1} o}$ coincide.
	Given $\partial^{\ell-1} v$ and $\hat\SIGMA_{\mathrm{p},\partial^{\ell-1} v}$ every vertex $u$ at distance precisely $\ell-1$ from $v$ in $\Gpara$ then requires another $d-1$ neighbours outside of $\partial^{\ell-1} v$.
	Because $\Gpara$ is uniformly random given $\fSpara$, for each $u$ these $d-1$ neighbours are simply the endpoints of edges $e_{u,1},\ldots,e_{u,d-1}$ drawn randomly from the set of all remaining edges with one endpoint of colour $\spara(u)$.
	Since $\ell=o(\log n)$, the subgraph $\partial^\ell v$ consumes no more than $n^{o(1)}$ edges.
	As a consequence, for each neighbour $w\not\in\partial^{\ell-1}v$ the colour $\spara(w)$ has distribution $\rho^\nu(\spara(u),\nix)$, up to an error of $n^{o(1)-1}$ in total variation.
	Finally, the probability that two vertices at distance precisely $\ell$ from $v$ are neighbours is bounded by $n^{o(1)-1}$ as well.

	By comparison, in the broadcasting process on $\TT_{d}$ the colours of the children of $y$ are always drawn independently from the distribution $\rhopara(\vsigma_{d,\beta,\nupara}(y),\nix)$.
	Hence, the colours of the vertices at distance $\ell$ in the two processes can be coupled to completely coincide with probability $1-O(d^\ell(\zeta+n^{o(1)-1}))$, as claimed. 
	
	In addition, since we work with the conditional Boltzmann distributions where we ``cut off'' a part of the phase space, we need to verify that the configuration is very unlikely  to hit the boundary of $\Spara$. To see this, recall from Proposition \ref{Prop_max} that, for $\beta \in (0, \KS)$,   $(\nupara, \rhopara )$ is a stable local maxima of $F_{d, \beta}$ i.e. there exist $\delta,c>0$ such that 
\begin{align}\label{eqstabmax1}
F_{d,\beta}(\nu',\rho')&\leq F_{d,\beta}(\nupara,\rhopara)-c\bc{\|\nupara -\nu'\|^2+\|\rhopara-\rho'\|^2}
\end{align}
for all $\nu'\in \cP([q])$ and $\rho'\in \cR(\nu')$ such that $\|\nupara-\nu'\|+\|\rhopara-\rho'\|<\delta$. Now, choose $\eps$ in the definition of $\Spara(\eps)$ such that $\eps>\delta$ and define  $T_p(\delta)=\cbc{\sigma \in [q]^{n}:\frac{1}{n} \sum_{c\in[q]} \abs{\sigma^{-1} (c)}= \nupara + \delta^\Delta} $ for some $\Delta>0$. Moreover, define a probability distribution $\nupara'$ on the $q$ colours by $\nupara'(c)=\frac{1}{q}+\frac{\delta^{\Delta}}{ q}$  for  all $c \in [q]$ and let $\rhopara' \in \cR(\nupara')$ the corresponding maximizer for $F_{d, \beta}(\nupara', \cdot)$
 (as in \ref{eqLemma_Jean}). Furthermore, choose $\Delta$ sufficiently small so that $ \|\nupara-\nupara'\|+\|\rhopara-\rhopara'\|<\delta $. Thus, by \eqref{eqstabmax1} and Lemma \ref{lemma_EZp} we have
 \begin{align*}
 	\pr\brk{\spara \in T_p(\delta) }&=\sum_{\sigma \in T_p(\delta) }\frac{\vecone\cbc{\sigma\in\Spara}\Erw[\emm^{\beta\cH_{\G}(\sigma)}]}{\Erw[\Zpara(\G)]} \leq \frac{\exp\bc{ n F_{d,\beta}(\nupara',\rhopara') } }{\exp \bc{ n F_{d,\beta}(\nupara,\rhopara) + O(\log n)} } \\
 	 & \leq \exp\bc{\bc{- c \bc{ \delta^{2 \Delta} +  \|\rhopara-\rho'\|^2 } + o(1) } n  } \leq \exp\bc{ \bc{-K +o(1)} n}
 \end{align*}
 for some sufficiently large constant $K$, as desired.
\end{proof}

The colouring of the neighbourhood of $v_1$ in $\Gferro$ admits a similar coupling with the ferromagnetic version of the broadcasting process.

\begin{lemma}\label{Lemma_nonrecFerro}
Let $d,q\geq3$ be integers and $\beta>\uniq$ be real.
Then, for any vertex $v$ and any non-negative integer $\ell=o(\log n)$, given $\fSferro$ \whp\ we have 
\begin{align*}
	\dTV(\hat\SIGMA_{\mathrm{f},\partial^\ell v},\tau_{\partial^\ell o})=O\bc{d^\ell\bc{\dTV(\nu^{\sferro},\nuferro)+\dTV(\rho^{\hG(\sferro),\sferro},\rhoferro)+n^{-0.99}}}.
\end{align*}
\end{lemma}
\begin{proof}
	The argument from the proof of \Lem~\ref{Lemma_nonrecPara} carries over directly.
\end{proof}

\begin{proof}[Proof of \Prop~\ref{Prop_nonreG}]
	We  prove the first statement concerning $\Gferro$; the proof of the second statement for $\Gpara$ is analogous.
	Due to \Prop~\ref{Prop_Nishi} we may work with the random graph $\hat\G(\sferro)$ with planted configuration $\sferro$.
	Fix an arbitrary vertex $v$ and $\ell=\lceil\log\log n\rceil$. For the first assertion, by the Nishimori identity, it suffices to prove that 
	\begin{align}\label{eqProp_nonreG1}
		\sum_{c\in[q]}\ex\abs{\nuferro(c)-\mu_{\hat\G(\sferro),\beta}(\vsigma_{v}=c\mid\vsigma_{\partial^\ell v}=\sferro{}_{,\partial^\ell v})}&<\ell^{-3},	
	\end{align}
	where the expectation is over the choice of the pair $(\hat\G(\sferro),\sferro)$. 
	Indeed, the desired $(o(1),\ell^{-1},\ell)$-non-reconstruction property follows from \eqref{eqProp_nonreG1} and Markov's inequality.

	To obtain \eqref{eqProp_nonreG1} we first apply \Lem~\ref{Lemma_sferro}, which implies that with probability $1-o(1/n)$,
	\begin{align}\label{eqProp_nonreG2}
	\dTV\bc{\nu^{\sferro},\nuferro}+\dTV\bc{\rho^{\hG(\sferro),\sferro},\rhoferro}\leq n^{-1/4}.
	\end{align}
	Further, assuming \eqref{eqProp_nonreG2}, we obtain from \Lem~\ref{Lemma_nonrecFerro} that
\begin{align}\label{eqProp_nonreG3}
	\dTV(\SIGMA_{\ferro,\partial^\ell v},\tau_{\partial^\ell o})=o(n^{-1/5}).
\end{align}
Hence, the colourings $\partial^\ell v$ and $\tau_{\partial^\ell o}$ can be coupled such that both are identical with probability $1-o(n^{-1/5})$.
Consequently, \eqref{eqProp_nonreG1} follows from \Prop~\ref{Prop_nonre}.

	Thus, we are left to prove the second assertion concerning $(\exp(-\chi n),\delta,\ell,\muferro)$-non-reconstruction.
	Hence, given $\delta>0$ pick a large enough $\ell=\ell(d,\beta,\delta)>0$, a small enough $\zeta=\zeta(\delta,\ell)>0$ and even smaller $\xi=\xi(\delta,\ell,\zeta)>0$, $\chi=\chi(d,\beta,\xi)>0$.
	Then in light of \Lem~\ref{Lemma_sferro} we may assume that
	\begin{align}\label{eqProp_nonreG5}
	\dTV\bc{\nu^{\sferro},\nuferro}+\dTV\bc{\rho^{\hG(\sferro),\sferro},\rhoferro}<\xi.
	\end{align}
	Further, let $\vX$ be the number of vertices $u$ such that 
	$$\sum_{c\in[q]}\abs{\nuferro(c)-\mu_{\hat\G(\sferro),\beta}(\vsigma_{u}=c\mid\vsigma_{\partial^\ell u}=\sferro{}_{,\partial^\ell u})}>\zeta.$$
	Then \Prop~\ref{Prop_nonre}, \eqref{eqProp_nonreG5} and \Lem~\ref{Lemma_nonrecFerro} imply that $\ex[\vX]<\zeta n$.
	Moreover, $\vX$ is tightly concentrated about its mean.
	Indeed, adding or removing a single edge of the random $d$-regular graph $\hat\G(\sferro)$ can alter the $\ell$-th neighbourhoods of no more than $d^\ell$ vertices.
	Therefore, the Azuma--Hoeffding inequality shows that 
		\begin{align}\pr\brk{\vX>\ex[\vX\mid\fSferro]+\zeta n\mid\fSferro}<\exp(-\chi n),\end{align}
	as desired.
\end{proof}

\subsection{Proof of \Cor~\ref{Cor_nonreG}}\label{Sec_Cor_nonreG}
We derive the corollary from \Prop~\ref{Prop_nonreG}, the Nishimori identity from \Prop~\ref{Prop_Nishi} and the formula~\eqref{eqsmmF} for the second moment.
As a first step we derive an estimate of the typical overlap of two configurations drawn from the Boltzmann distribution.
To be precise, for a graph $G=(V,E)$, the overlap of two configurations $\sigma,\sigma'\in[q]^{V}$ is defined as the probability distribution $\nu(\sigma,\sigma')\in\cP([q]^2)$ with
\begin{align*}
	\nu_{c,c'}(\sigma,\sigma')&=\frac1n\sum_{v\in V(G)}\vecone\cbc{\sigma_{v}=c,\,\sigma'_{v}=c'}&&(c,c'\in[q]).
\end{align*}
Thus, $\nu(\sigma,\sigma')$  
gauges the frequency of the colour combinations among vertices. 

\begin{lemma}\label{Lemma_overlapPara}
	Let $d,q\geq3$ be integers and $\beta<\KS$ be real. Let $\sparaGpara,\sparaGpara'$ be independent samples from $\mu_{\Gpara,\beta}(\nix\mid\Spara)$.
	Then $\Erw\brk{\dTV\big(\nu(\sparaGpara,\sparaGpara'),\nupara\tensor\nupara\big)}=o(1)$.
\end{lemma}
\begin{proof}
	Due to the Nishimori identity \eqref{eqProp_Nishi} it suffices to prove that \whp{} for a sample $\SIGMA_{\hG(\spara)}$ from $\mu_{\hG(\spara),\beta}(\nix\mid\Spara)$ it holds that
	\begin{align}\label{eqLemma_overlapPara}
\dTV\big(\nu(\spara,\SIGMA_{\hG(\spara),\mathrm{p}}),\nupara\tensor\nupara\big)&=o(1)
	\end{align}
	To see \eqref{eqLemma_overlapPara}, for colours $s,t\in[q]$, we consider  the first and second moment of the number of vertices $u$ with $\spara(u)=s$ and $\SIGMA_{\hG(\spara),\mathrm{p}}(u)=t$. To facilitate the analysis of the second moment, it will be convenient to consider the following configuration $\SIGMA_{\hG(\spara),\mathrm{p}}'$.  Let $\vec v,\vec w$ be  two random vertices such that $\spara(\vec v)=\spara(\vec w)=s$.
	Also let $\ell=\ell(n)=\lceil\log\log n\rceil$.
	Now, draw $\SIGMA_{\hG(\spara),\mathrm{p}}''$ from $\mu_{\hG(\spara),\beta}(\nix\mid\Spara)$ and subsequently generate $\sparaGpara'$ by re-sampling the colours of the vertices at distance less than $\ell$ from $\vec v,\vec w$ given the colours of the vertices at distance precisely $\ell$ from $\vec v,\vec w$ and the event $\Spara$.
	Then $\SIGMA_{\hG(\spara),\mathrm{p}}'$  has distribution $\mu_{\hG(\spara),\beta}(\nix\mid\Spara)$. 	Moreover, since for two random vertices $\vec v, \vec w$ their $\ell$-neighbourhoods are going to be disjoint,  \Prop~\ref{Prop_nonreG} implies that \whp
	\begin{align}\label{eqLemma_overlapPara1}
		\pr\brk{\sparaGpara'(\vec v)=\chi,\,\sparaGpara'(\vec w)=\chi'\mid \spara,\hG(\spara),\vec v,\vec w}&=\nupara(\chi)\nupara(\chi')+o(1)\qquad\mbox{for all }\chi,\chi'\in[q].
	\end{align}
	Hence, for a colour $t\in[q]$ let $\vX(s,t)$ be the number of vertices $u$ with $\spara(u)=s$ and $\sparaGpara'(u)=t$.
	Then \eqref{eqLemma_overlapPara1} shows that \whp
	\begin{align*}
		\ex\brk{\vX(s,t)\mid\spara,\hG(\spara)}\sim\frac{n}{q^2},\qquad\ex\brk{\vX(s,t)^2\mid\spara,\hG(\spara)}&\sim\frac{n^2}{q^4}.
	\end{align*}
	Thus, \eqref{eqLemma_overlapPara} follows from Chebyshev's inequality.
\end{proof}

\begin{lemma}\label{Lemma_overlapFerro}
	Let $d,q\geq3$ be integers and $\beta>\uniq$ be real. Let $\sferroGferro,\sferroGferro'$ be independent samples from $\mu_{\Gferro,\beta}(\nix\mid\Sferro)$.
	Then $\Erw\brk{\dTV\big(\nu(\sferroGferro,\sferroGferro'),\nuferro\tensor\nuferro\big)}=o(1)$.
\end{lemma}
\begin{proof}
	The same argument as in the proof of \Lem~\ref{Lemma_overlapPara} applies.
\end{proof}

We proceed to apply the second moment method to truncated versions of the paramagnetic and ferromagnetic partition functions $\Zpara,\Zferro$ where we expressly drop graphs that violate the overlap bounds from \Lem s~\ref{Lemma_overlapPara} and~\ref{Lemma_overlapFerro}.
Thus, we introduce
\begin{align}\label{eqYpara}
	\Ypara(G)&=\Zpara(G)\cdot\vecone\cbc{\ex\big[\dTV(\nu(\SIGMA_{G,\mathrm{p}},\SIGMA_{G,\mathrm{p}}'),\nupara\tensor\nupara)\big]=o(1)},\\
	\Yferro(G)&=\Zferro(G)\cdot\vecone \cbc{\ex\big[\dTV(\nu(\SIGMA_{G,\mathrm{f}},\SIGMA_{G,\mathrm{f}}'),\nuferro\tensor\nuferro)\big]=o(1)}.
\label{eqYferro}
\end{align}
Estimating the second moments of these two random variables is a cinch because by construction we can avoid an explicit optimisation of the function $F^\tensor_{d,\beta}$ from \eqref{eqsmmF}.
Indeed, because we drop graphs $G$ whose overlaps stray far from the product measures $\nupara\tensor\nupara$ and $\nuferro\tensor\nuferro$, respectively, we basically just need to evaluate the function $F^\tensor_{d,\beta}$ at $\nupara\tensor\nupara$ and $\nuferro\tensor\nuferro$.

\begin{corollary}\label{Cor_overlap}
Let $d\geq3$.
\begin{enumerate}[(i)]
	\item If $\beta<\KS$, then $\ex[\Ypara(\G)]\sim\ex[\Zpara(\G)]$ and $\ex[\Ypara(\G)^2]\leq \exp(o(n))\ex[\Zpara(\G)]^2$.
	\item If $\beta>\uniq$, then $\ex[\Yferro(\G)]\sim\ex[\Zferro(\G)]$ and $\ex[\Yferro(\G)^2]\leq \exp(o(n))\ex[\Zferro(\G)]^2$.
\end{enumerate}
\end{corollary}
\begin{proof}
	Assume that $\beta<\KS$.
	Let $\cE_{\mathrm p}=\{G:\ex\big[\dTV(\nu(\SIGMA_{G,\mathrm{p}},\SIGMA_{G,\mathrm{p}}'),\nupara\tensor\nupara)\big]=o(1)\}$.
	Combining \Lem~\ref{Lemma_overlapPara} with the Nishimori identity \eqref{eqProp_Nishi}, we obtain
	\begin{align}\label{eqCor_overlap1}
		\frac{\ex[\Ypara(\G)]}{\ex[\Zpara(\G)]}&=\pr\brk{\Gpara\in\cE_{\mathrm p}}\sim1
	\end{align}
	and thus $\ex[\Ypara(\G)]\sim\ex[\Zpara(\G)]$.

	Regarding the second moment, consider the set $\cP_{\para}(n)$ of all probability distributions $\nu$ on $[q]\times[q]$ such that $n\nu(\chi,\chi')$ is an integer for all $\chi,\chi'\in[q]$ and such that $\dTV(\nu,\vec u)=o(1)$.
	Let $\cR_\para(\nu,n)$ be the set of all distributions $\rho$ on $[q]^4$ such that 
\begin{align}\label{eqConstraint1}
	\rho(\chi,\chi',\chi'',\chi''')&=\rho(\chi'',\chi''',\chi,\chi')&&\mbox{for all }\chi,\chi',\chi'',\chi'''\in[q]\quad\mbox{ and}\\
	\sum_{\chi'',\chi'''\in[q]}\rho(\chi,\chi',\chi'',\chi''')&=\nu(\chi,\chi')&&\mbox{for all }\chi,\chi'\in[q]\label{eqConstraint2}
	\end{align}
	and such that $n\rho(\chi,\chi',\chi'',\chi''')$ is an integer for all $\chi,\chi',\chi'',\chi'''\in[q]$.
	Using the definition \eqref{eqYpara} of $\Ypara$, \Lem~\ref{Lemma_Samuel} and the linearity of expectation, we bound
	\begin{align}\nonumber
		\ex\brk{\Ypara(\G)^2}&\leq(1+o(1))\sum_{\sigma,\sigma'\in[q]^n}\vecone\cbc{\dTV(\nu(\sigma,\sigma'), \nupara\tensor\nupara)=o(1)}
			\ex\brk{\emm^{\beta(\cH_{\G}(\sigma)+\cH_{\G}(\sigma'))}}\\
							 &\leq\sum_{\nu\in\cP_\para(n)}\binom{n}{\nu n}\sum_{\rho\in\cR_\para(\nu,n)}
								\exp\Big[\frac{dn}2\sum_{\chi,\chi',\chi'',\chi'''=1}^q\rho(\chi,\chi',\chi'',\chi''')\log\frac{\nu(\chi,\chi')\nu(\chi'',\chi''')}{\rho(\chi,\chi',\chi'',\chi''')}\nonumber\\
							 &\qquad\qquad\qquad\qquad\qquad\qquad\qquad\qquad\qquad\qquad+\beta\bc{\vecone\cbc{\chi=\chi''}+\vecone\cbc{\chi'=\chi'''}}+O(\log n)\Big].\label{eqCor_overlap2}
	\end{align}
	For any given $\nu$ the term inside the square brackets is a strictly concave function of $\rho$.
	Therefore, for any $\nu$ there exists a unique maximiser $\rho^*_\nu$.
	Moreover, the set $\cR_\para(\nu,n)$ has size $|\cR_\para(\nu,n)|=n^{O(1)}$.
	Hence, using Stirling's formula we can simplify \eqref{eqCor_overlap2} to
	\begin{align}\nonumber
		\ex\brk{\Ypara(\G)^2}&\leq\sum_{\nu\in\cP_\para(n)}
								\exp\Big[-n\sum_{\chi,\chi'=1}^q\nu(\chi,\chi')\log\nu(\chi,\chi')+\frac{dn}2\sum_{\chi,\chi',\chi'',\chi'''=1}^q\rho^*_\nu(\chi,\chi',\chi'',\chi''')\log\frac{\nu(\chi,\chi')\nu(\chi'',\chi''')}{\rho^*_\nu(\chi,\chi',\chi'',\chi''')}\nonumber\\
							 &\qquad\qquad\qquad\qquad\qquad\qquad\qquad\qquad\qquad\qquad+\beta\bc{\vecone\cbc{\chi=\chi''}+\vecone\cbc{\chi'=\chi'''}}+O(\log n)\Big].\label{eqCor_overlap3}
	\end{align}
	To further simplify the expression notice that the maximiser $\rho^*_\nu$ is the unique solution to a concave optimisation problem subject to the linear constraints \eqref{eqConstraint1}--\eqref{eqConstraint2}.
	Since the constraints \eqref{eqConstraint2} themselves are linear in $\nu$, by the inverse function theorem the maximiser $\rho^*_\nu$ is a continuous function of $\nu$.
	In effect, since $|\cP_\para(n)|=n^{O(1)}$,  we can bound \eqref{eqCor_overlap3} by the contribution of the uniform distribution $\nupara\tensor\nupara$ only.
	We thus obtain
	\begin{align}\nonumber
		\ex\brk{\Ypara(\G)^2}&\leq q^n\exp\Big[\frac{dn}2\sum_{\chi,\chi',\chi'',\chi'''=1}^q\rho^*_{\nupara\tensor\nupara}(\chi,\chi',\chi'',\chi''')\log\frac{\nupara(\chi,\chi')\nupara(\chi'',\chi''')}{\rho^*_{\nupara\tensor\nupara}(\chi,\chi',\chi'',\chi''')}\nonumber\\
							 &\qquad\qquad\qquad\qquad\qquad\qquad\qquad\qquad\qquad\qquad+\beta\bc{\vecone\cbc{\chi=\chi''}+\vecone\cbc{\chi'=\chi'''}}+o(n)\Big].\label{eqCor_overlap4}
	\end{align}

	Finally, the maximiser $\rho^*_{\nupara\tensor\nupara}$ in \eqref{eqCor_overlap4} works out to be $\rho^*_{\nupara\tensor\nupara}=\rhopara\tensor\rhopara$.
	To see this, recall from \Lem~\ref{Lemma_overlapPara} that $\nupara$ is the uniform distribution on $[q]$. 
	It therefore remains to show that subject to \eqref{eqConstraint1}--\eqref{eqConstraint2}, the function
	\begin{align*}
		g(\rho)&=\sum_{\chi,\chi',\chi'',\chi'''=1}^q\rho(\chi,\chi',\chi'',\chi''')\log\frac{\nupara(\chi) \nupara(\chi') \nupara(\chi'') \nupara(\chi''')}{\rho(\chi,\chi',\chi'',\chi''')} +\beta\bc{\vecone\cbc{\chi=\chi''}+\vecone\cbc{\chi'=\chi'''}}\\
			   &=-4\log q-\sum_{\chi,\chi',\chi'',\chi'''=1}^q\rho(\chi,\chi',\chi'',\chi''')\log\bc{\rho(\chi,\chi',\chi'',\chi''')} -\beta\bc{\vecone\cbc{\chi=\chi''}+\vecone\cbc{\chi'=\chi'''}}
	\end{align*}
	attains its maximum at the distribution $\rho=\rhopara\tensor\rhopara$. 
	Since $g$ is strictly concave, the unique maximum occurs at the unique stationary point of the Lagrangian
	\begin{align*}
		L_{\para}&=g(\rho)+\sum_{\chi,\chi',\chi'',\chi'''}\lambda_{\chi,\chi',\chi'',\chi'''}\bc{\rho\bc{\chi,\chi',\chi'',\chi'''}-\rho\bc{\chi'',\chi''',\chi,\chi'}} \\
		&+\sum_{\chi,\chi'}\lambda_{\chi,\chi'}\bc{\sum_{\chi'',\chi'''\in[q]}\rho(\chi,\chi',\chi'',\chi''')-\nu(\chi,\chi')}.
		\end{align*}
	Since the derivatives work out to be
	\begin{align*}
		\frac{\partial L_{\para}}{\partial\rho(\chi,\chi',\chi'',\chi''')}&=-1-\log\rho(\chi,\chi',\chi'',\chi''') +\lambda_{\chi,\chi',\chi'',\chi'''}-\lambda_{\chi'',\chi''',\chi',\chi'} +\lambda_{\chi,\chi'}+\beta\vecone\{\chi=\chi''\}+\beta\vecone\{\chi'=\chi'''\},
	\end{align*}
	for the choice $\rho=\rhopara\tensor\rhopara$ there exist Lagrange multipliers such that all partial derivatives vanish.

	The proof of (ii) proceeds analogously.
\end{proof}

\begin{proof}[Proof of \Cor~\ref{Cor_nonreG}]
The corollary is now an immediate consequence of \Cor~\ref{Cor_overlap}, the Paley-Zygmund and Azuma inequalities.
\end{proof}

\subsection{Proof of \Cor~\ref{Cor_cutm}}\label{Sec_Cor_cutm}
To prove \Cor~\ref{Cor_cutm} we derive the following general transfer principle from the estimate of the Boltzmann weights of $\Sferro$ and $\Spara$ from \Cor~\ref{Cor_nonreG}.

\begin{lemma}\label{Cor_pseudo}
Let $d\geq3$.
\begin{enumerate}[(i)]
	\item If $\beta<\KS$, then for any event $\cE$ with $\pr\brk{\Gpara\in\cE}\leq \exp(-\Omega(n))$ we have $\pr\brk{\GG\in\cE}=o(1)$.
	\item If $\beta>\uniq$, then for any event $\cE$ with $\pr\brk{\Gferro\in\cE}\leq \exp(-\Omega(n))$ we have $\pr\brk{\GG\in\cE}=o(1)$.
\end{enumerate}
\end{lemma}
\begin{proof}
	This follows from a ``quiet planting'' argument akin to the one from~\cite{Barriers}.
	Specifically, \Thm~\ref{Thm_Z} and \Prop~\ref{Prop_max} show that for $\beta<\KS$ the event $\cZ_\para=\{\Zpara(\G)=\ex[\Zpara(\G)]\exp(o(n))\}$ occurs \whp\
	Therefore, recalling the definition \eqref{eqDefPlanted} of the planted model, we obtain
	\begin{align}\nonumber
		\pr\brk{\G\in\cE}&\leq\pr\brk{\G\in\cE\cap\cZ_\para}+\pr\brk{\G\not\in\cZ_\para}\leq\frac{\ex[\vecone\{\G\in\cE\}\Zpara(\G)]\exp(o(n))}{\ex[\Zpara(\G)]}+o(1)\\
							&\leq\exp(o(n))\pr\brk{\Gpara\in\cE}+o(1)=o(1).\label{eqCor_pseudo1}
		\end{align}
	Since the simple random regular graph $\GG$ is contiguous with respect to $\G$, assertion (i) follows from \eqref{eqCor_pseudo1}.
	The proof of (ii) is identical.
\end{proof}

\begin{proof}[Proof of \Cor~\ref{Cor_cutm}]
The assertion follows from \Lem~\ref{Cor_pseudo} and \Prop~\ref{Prop_nonreG}.
\end{proof}

\section{Metastability and Slow mixing}\label{sec:slowSW}
In this section, we prove Theorems~\ref{Thm_meta} and~\ref{thm:SWslow}. Recall from Section~\ref{sec:states} the paramagnetic and ferromagnetic states $\Sparad$ and $\Sferrod$ for $\eps>0$. For the purposes of this section we will need to be more systematic of keeping track the dependence of these phases on $\eps$. In particular, we will use the more explicit notation $\Zparad(G)$ and $\Zferrod(G)$ to denote  the quantities $\Zpara(G)$ and $\Zferro(G)$, respectively,  from \eqref{eqZfp}.

 The following lemma reflects the fact that $\nupara$ and $\nuferro$ are local maxima of the first moment.
\begin{lemma}\label{lem:phasebounds}
Let $q,d\geq 3$ be integers and $\beta>0$ be real. Then, for  all sufficiently small constants $\eps'>\eps>0$ and any constant $\theta>0$, there exists constant $\zeta>0$ such that \whp{} over $G\sim \G$, it holds that 
\begin{enumerate}
\item \label{it:phasepara} If $\beta<\KS$, then $\Zparad(G)= \emm^{ o(n)}\Erw[\Zparad(\G)]$ and $\Zparadp(G)\leq (1+\emm^{-\zeta n})\Zparad(G)$.
\item \label{it:phaseferro} If $\beta>\uniq$, then $\Zferrod(G)= \emm^{o(n)}\Erw[\Zferrod(\G)]$ and $\Zferrodp(G)\leq (1+\emm^{-\zeta n})\Zferrod(G)$.
\end{enumerate}
\end{lemma}
\begin{proof}
We first prove Item~\ref{it:phasepara}, let $\beta<\KS$.  Recall from \eqref{eqFirstMmtBethe} the function $F(\nu,\rho):=F_{d, \beta}(\nu,\rho)$ for $\nu\in \cP([q])$ and $\rho\in\cR(\nu)$. By Corollary~\ref{Cor_nonreG}, for all sufficiently small constant $\eps>0$, we have that 
\[\Erw[\tfrac{1}{n}\log \Zparad(\G)]=\cB_{d,\beta}(\mupara)+o(1)=F(\nupara,\rhopara)+o(1),\] where the last equality holds by \Lem~\ref{Lemma_Jean}. Applying Azuma's inequality to the random variable $\log \Zparad(\G)$ by revealing the edges of $\G$ one-by-one,  we therefore obtain that \whp{} it holds that $\Zparad(G)= \emm^{nF(\nupara,\rhopara)+o(n)}$. Also, from Lemma~\ref{lemma_EZp} we have that $\Erw[\Zparad(\G)]= \emm^{nF(\nupara,\rhopara)+o(n)}$, so we obtain that 
$\Zparad(G)= \emm^{o(n)}\Erw[\Zparad(\G)]$
proving the first inequality of Item~\ref{it:phasepara}. For the second inequality, recall from  Proposition~\ref{Prop_max} that $(\nupara,\rhopara)$ is a local maximum of $F$ for $\beta<\KS$, cf. \eqref{eqstabmax}. Therefore, for all sufficiently small constants $\eps'>\eps>0$,   there exists constant $\zeta>0$ such that 
\begin{equation}\label{eq:localineq}
F(\nu,\rho)\leq F(\nupara,\rhopara)-4\zeta
\end{equation}
for all $\nu\in \cP([q])$ and $\rho\in \cR(\nu)$ with
\begin{equation}\label{eq:ring}
\eps\leq \norm{\nu-\nupara}+ \norm{\rho-\rhopara}\leq \eps'. 
\end{equation}
Using \eqref{eq:summingsigma}, we see that 
\[\Erw\big[\Zparadp(\G)-\Zparad(\G)\big]\leq\sum_{\nu, \rho}\exp(nF(\nu,\rho)+O(\log n))\]
where the sum ranges over $\nu\in \cP([q])$ and $\rho\in\cR(\nu)$ satisfying \eqref{eq:ring} such that $n\nu(s),dn\rho(s,t)$ are integers for all $s,t\in[q]$, and $dn\rho(s,s)$ is even. Since there are at most $n^{O(1)}$ choices for  such colour statistics $\nu,\rho$, we obtain that	$\Erw\big[\Zparadp(\G)-\Zparad(\G)\big]\leq \emm^{n(F(\nupara,\rhopara)-3\zeta)}$ for all sufficiently large $n$. 
By Markov's inequality, we therefore have that \whp{} $\Zparadp(G)-\Zparad(G)\leq \emm^{n F(\nupara,\rhopara)-2\zeta n}$. As we showed above, \whp{} $\Zparad(G)= \emm^{nF(\nupara,\rhopara)+o(n)}$, so combining these we obtain that \whp{} $\Zparad(G)\geq \emm^{\zeta n}\big(\Zparadp(G)-\Zparad(G)\big)$,
completing the proof of Item~\ref{it:phasepara} of the lemma. 

For the second item of the lemma, the proof is completely analogous, using the fact from Proposition~\ref{Prop_max} that $(\nuferro,\rhoferro)$ is a local maximum of $F(\nu,\rho)$ for $\beta>\uniq$.
\end{proof}

 \Thm~\ref{Thm_meta} will follow by way of a conductance argument. Let $G=(V,E)$ be a graph, and $P$ be the transition matrix for the Glauber dynamics defined in Section \ref{sec_Metastability}. 
For a set $S \subseteq [q]^{V}$ define the \emph{bottleneck ratio} of $S$ to be 
\begin{equation}\label{eq_bottleneckratio} 
	\Phi\bc{S}= \frac{\sum_{\sigma\in S,\,\tau\not\in S}\mu_{G,\beta}(\sigma)P(\sigma,\tau)}{\mu_{G,\beta}(S)}
\end{equation}
The following lemma provides a routine conductance bound (e.g., \cite[\Thm~7.3]{LPW}).
For the sake of completeness the proof is included in Appendix~\ref{sec_lem_bottleNeck}.
\newcommand{\statelembottle}{Let $G=(V,E)$ be a graph. For any $S \subseteq [q]^{V}$ such that $\mu_{G}(S)>0$ and any integer $t\geq 0$ we have $\norm{ \mu_{G,S} P^t - \mu_{G,S}}_{TV} \leq t \Phi(S).$}
\begin{lemma} \label{lem_bottleNeck} \statelembottle
\end{lemma}
\begin{proof}[Proof of Theorem~\ref{Thm_meta}]
We prove the statement for the pairing model $\G$, the result for $\GG$ follows immediately by contiguity.   Let $\eps'>\eps>0$ and $\zeta>0$ be small constants such that Lemma~\ref{lem:phasebounds} applies, and let $G\sim \G$ be a graph satisfying the lemma. Set for convenience $\mu=\mu_{\G,\beta}$; we consider first the metastability of $\Sferrod$ for $\beta>\uniq$.

Since Glauber updates one vertex at a time it is impossible in one step to move from $\sigma\in \Sferrod$ to  $\tau\in [q]^n\backslash \Sferrodp$, i.e., $P(\sigma,\tau)=0$, and therefore
\begin{align*}\Phi\big(\Sferrod\big)&=\frac{\sum_{\sigma\in \Sferrod}\sum_{\tau\notin \Sferrod}\mu(\sigma)P(\sigma,\tau)}{\mu\big(\Sferrod\big)}= \frac{\sum_{\sigma\in \Sferrod}\sum_{\tau\in \Sferrodp\backslash\Sferrod}\mu(\sigma)P(\sigma,\tau)}{\mu\big(\Sferrod\big)}
\end{align*}
By reversibility of Glauber, for any $\sigma,\tau\in [q]^n$ we have $\mu(\sigma)P(\sigma,\tau)=\mu(\tau)P(\tau,\sigma)$, and therefore
\[\sum_{\sigma\in \Sferrod}\sum_{\tau\in \Sferrodp\backslash\Sferrod}\mu(\sigma)P(\sigma,\tau)=\sum_{\tau\in \Sferrodp\backslash\Sferrod}\mu(\tau)\sum_{\sigma\in \Sferrod}P(\tau,\sigma)\leq  \sum_{\tau\in \Sferrodp\backslash\Sferrod}\mu(\tau)=\mu\big(\Sferrodp\backslash\Sferrod\big)\]
Hence, $\Phi\big(\Sferrod\big)\leq \frac{\mu\big(\Sferrodp\backslash\Sferrod\big)}{\mu\big(\Sferrod\big)}=\tfrac{\Zferrodp(G)-\Zferrod(G)}{\Zferrod(G)}\leq \emm^{-\zeta n}$, where the last inequality follows from the fact that $G$ satisfies Lemma~\ref{lem:phasebounds}.  Lemma \ref{lem_bottleNeck} therefore ensures that for all nonnegative integers $T\leq \emm^{\zeta n/3}$
\begin{align} 
\label{eq_bottleneckferro}
\norm{ \mu\big(\nix\mid\Sferrod\big) P^{T} - \mu\big(\nix\mid\Sferrod\big)}_{TV} \leq T \cdot\Phi(\Sferro)\leq \emm^{-2\zeta n/3}.
\end{align}
Now, consider the Glauber dynamics $(\sigma_t)_{t\geq0}$ launched from $\sigma_0$ drawn from $\mu_{\GG,\beta,\Sferrod}$, and denote by $T_{\escf}=\min\cbc{t >0: \sigma_t \notin \Sferrod}$ its escape time from $\Sferrod$. Observe that $\sigma_t$ has the same distribution as $\mu(\nix\mid\Sferrod) P^{t}$, so \eqref{eq_bottleneckferro}  implies that for all nonnegative integers $T\leq \emm^{\zeta n/3}$
\[	\big|\pr\brk{\sigma_T\in \Sferrod}-1\big|<\emm^{-2\zeta n/3}, \mbox{ or equivalently } \pr\brk{\sigma_T\notin \Sferrod}\leq \emm^{-2\zeta n/3}.\]
By a union bound over the values of $T$, we therefore obtain that $\pr[T_{\escf}\leq  \emm^{\zeta n/3}]\leq  \emm^{-\zeta n/3}$, thus proving that $\Sferrod$ is a metastable state for $\beta>\uniq$. Analogous arguments show that $\Sparad$ is a metastable state for $\beta<\KS$.

The slow mixing of Glauber for $\beta>\uniq$ follows from the metastability of $\Sferrod$. In particular, from Theorem~\ref{Thm_Boltzmann} we have that $\norm{\mu\big(\nix\mid\Sferrod\big)-\mu}\geq 3/5$ and therefore, from \eqref{eq_bottleneckferro},   $\norm{ \mu\big(\nix\mid\Sferrod\big) P^{T}-\mu}\geq 1/2$, yielding that the mixing time is $\emm^{\Omega(n)}$.
\end{proof}

The final ingredients to establish Theorem~\ref{thm:SWslow} are the following results, bounding the probability that Swendsen-Wang escapes $\Sparad$ and $\Sferrod$. More precisely, for a graph $G$, a configuration $\sigma\in [q]^n$, and $S\subseteq [q]^n$, let $P^{G}_{SW}(\sigma\to S)$ denote the probability that after one step of SW on $G$ starting from $\sigma$, we end up in a configuration in $S$.

The following proposition shows that for almost all pairs $(G,\sigma)$ from the paramagnetic planted distribution $\big(\hat\G\big(\sparad\big),\sparad\big)$, the probability that SW leads to a configuration in the paramagnetic phase, slightly enlarged, is $1-\emm^{-\Omega(n)}$.
\newcommand{\statepropSWpara}{Let $q,d\geq 3$ be integers and $\beta\in (\uniq,\KS)$ be real. Then, for all sufficiently small constants $\eps'>\eps>0$, there exists constant $\eta>0$ such that with probability $1-\emm^{- \eta n}$ over the planted distribution $(G,\sigma)\sim \big(\hat\G\big(\sparad\big),\sparad\big)$, it holds that $P^{G}_{SW}\big(\sigma\to\Sparadp\big)\geq  1-\emm^{-\eta n}$.}
\begin{proposition} \label{thm:SWpara}
\statepropSWpara
\end{proposition}
The following  establishes the analogue of the previous proposition for the ferromagnetic planted distribution $\big(\hat\G\big(\sferrod\big),\sferrod\big)$. Note here that SW might change the dominant colour due to recolouring step, so,  for $\eps>0$, we now need to consider the set of configurations $\tSferrod$ that consists of the ferromagnetic phase $\Sferrod$ together with its $q-1$ permutations,  and the probability that SW escapes from it, starting from a ferromagnetic state.
\newcommand{\statepropSWferro}{Let $q,d\geq 3$ be integers and $\beta\in (\uniq,\KS)$ be real. Then, for all sufficiently small constants $\eps'>\eps>0$, there exists  constant $\eta>0$ such that  with probability $1-\emm^{-\eta n}$ over the planted distribution $(G,\sigma)\sim \big(\hat\G\big(\sferrod\big),\sferrod\big)$, it holds that $P^{G}_{SW}\big(\sigma\to \tSferrodp\big)\geq 1-\emm^{-\eta n}$.}
\begin{proposition}\label{thm:SWferro}
\statepropSWferro
\end{proposition}

\begin{proof}[Proof of Theorem~\ref{thm:SWslow}]
We prove the statement for the pairing model $\G$, the result for $\GG$ follows immediately by contiguity. We consider first the metastability for the ferromagnetic phase when $\beta>\uniq$. Let $\eps'>\eps>0$ and $\eta,\zeta>0$ be small constants such that Lemma~\ref{lem:phasebounds} and Propositions~\ref{thm:SWpara},~\ref{thm:SWferro} all apply. Let $\theta=\tfrac{1}{10}\min\{\eta,\zeta\}$.

 Let $\cQ$ be the set of $d$-regular (multi)graphs that satisfy \[\Zferrod(G)\geq  \emm^{-\theta n}\Erw[\Zferrod(\G)]\mbox{ and } \Zferrodp(G)\leq (1+\emm^{-\zeta n})\Zferrod(G),\] and note that by Item~\ref{it:phaseferro} of Lemma~\ref{lem:phasebounds} it holds that $\Pr[\G\in \cQ]=1-o(1)$.
Moreover, let $\cQ'$ be the set of $d$-regular (multi)graphs $G$ such that the set of configurations where SW has conceivable probability of escaping $\tSferrodp$ has small weight, i.e., the set 
\[S_{\mathrm{Bad}}(G)=\big\{\sigma\in \tSferrod\,\big|\, P^{G}_{SW}\big(\sigma\to\tSferrodp\big)<1- \emm^{-\eta n}\big\}\]
has aggregate weight $Z_{\mathrm{Bad}}(G)=\sum_{\sigma\in S_{\mathrm{Bad}}(G)}\emm^{\beta \cH(G)}$ less than $\emm^{-\theta n}\Zferrod(G)$. We claim that for a $d$-regular graph $G$ such that $G\in \cQ\cap \cQ'$, it holds that $\Phi_{SW}\big(\tSferrod\big)\leq10\emm^{-\eta n}$, where $\Phi_{SW}(\cdot)$ denotes the bottleneck ratio for the SW-chain. Indeed, we have 
\begin{align*}
\Phi_{SW}\big(\tSferrod\big)&=\frac{\sum_{\sigma\in \tSferrod}\mu(\sigma)P^G_{SW}(\sigma\rightarrow [q]^n\backslash \tSferrod)}{\mu\big(\tSferrod\big)}
\leq \frac{\mu\big(S_{\mathrm{Bad}}(G)\big)+\sum_{\sigma\in \tSferrod\backslash S_{\mathrm{Bad}}(G)}\mu(\sigma)P^G_{SW}(\sigma\rightarrow [q]^n\backslash \tSferrod)}{\mu\big(\tSferrod\big)}
\end{align*}
We can decompose the sum in the numerator of the last expression as 
\[\sum_{\sigma\in \tSferrod\backslash S_{\mathrm{Bad}}(G)}\mu(\sigma)P^G_{SW}\big(\sigma\rightarrow [q]^n\backslash \tSferrodp\big)+\sum_{\sigma\in \tSferrod\backslash S_{\mathrm{Bad}}(G)}\mu(\sigma)P^G_{SW}\big(\sigma\rightarrow \tSferrodp\backslash \tSferrod\big).\]
For $\sigma\in \tSferrod\backslash S_{\mathrm{Bad}}(G)$, we have  $P^G_{SW}\big(\sigma\rightarrow [q]^n\backslash \tSferrodp\big)\leq \emm^{-\eta n}$ and therefore the first sum is upper bounded by $\emm^{-\eta n}\mu\big(\tSferrod\big)$. The second sum, using the reversibility of the SW chain, is upper bounded by $\mu\big(\tSferrodp\backslash \tSferrod\big)$.  Using these, we therefore have that
\begin{align*}
\Phi_{SW}\big(\tSferrod\big)&\leq \frac{\mu\big(S_{\mathrm{Bad}}(G)\big)+\emm^{-\eta n}\mu\big(\tSferrod\big)+\mu\big(\tSferrodp\backslash \tSferrod\big)}{\mu\big(\tSferrod\big)}\leq 10\emm^{-\theta n},
\end{align*}
since  $\frac{\mu\big(S_{\mathrm{Bad}}(G)\big)}{\mu\big(\tSferrod\big)}=\frac{Z_{\mathrm{Bad}}(G)}{q\Zferrod(G)}\leq \emm^{-\theta n}$ from the assumption $G\in \cQ'$ and $\frac{\mu\big(\tSferrodp\backslash \tSferrod\big)}{\mu\big(\tSferrod\big)}=\frac{q(\Zferrodp(G)-\Zferrod(G))}{q\Zferrod(G)}\leq \emm^{-\theta n}$ from Lemma~\ref{lem:phasebounds}. By arguments analogous to those in the proof of Theorem~\ref{Thm_meta}, we have that $\tSferrod$ is a metastable state for graphs $G\in \cQ\cap \cQ'$. Therefore, to finish the metastability proof for the random graph, it suffices to show that $\Pr[\G\in \cQ\cap \cQ']=1-o(1)$.

To do this, let $\cG(n,d)$ be the set of all multigraphs that can be obtained in the pairing model and $\Lambda_{d,\beta}(n)=\big\{(G,\sigma)\, \big| \, G\in \cG(n,d),\ \sigma\in \tSferrod\big\}$. Let $\cE$  be the  pairs $(G,\sigma)\in \Lambda_{d,\beta}(n)$ where one step of SW starting from  $G,\sigma$ stays within $\tSferrodp$ with probability $1- \emm^{-\Omega(n)}$, more precisely 
\[\cE=\Big\{(G,\sigma)\in \Lambda_{d,\beta}(n)\, \big| \, P^{G}_{SW}\big(\sigma\to\tSferrodp\big)\geq  1- \emm^{-\eta n}\Big\}.\]
The aggregate weight corresponding to  pairs $(G,\sigma)$ that do not belong to $\cE$ can be lower-bounded by
\begin{equation*}
\sum_{(G,\sigma)\in \Lambda_{d,\beta}\backslash \cE} \emm^{\beta \cH_G(\sigma)}\geq \sum_{\substack{(G,\sigma)\in \Lambda_{d,\beta}\backslash \cE;\\ G\in \cQ\backslash \cQ'}} \emm^{\beta \cH_G(\sigma)}=\sum_{G\in \cQ\backslash \cQ'} \sum_{\sigma\in \Sigma_{\mathrm{Bad}}(G)}\emm^{\beta \cH_G(\sigma)}\geq \emm^{-\theta n}\sum_{G\in \cQ\backslash \cQ'}\Zferrod(G).
\end{equation*}
For graphs $G\in\cQ$ we have $\Zferrod(G)\geq \emm^{-\theta n}\Erw[\Zferrod(\G)]$, and therefore
\begin{equation}\label{eq:inter14445}
\sum_{(G,\sigma)\in \Lambda_{d,\beta}\backslash \cE} \emm^{\beta \cH_G(\sigma)}\geq \emm^{-2\theta n} \big|\cQ\backslash \cQ'\big|\ \Erw\big[\Zferrod(\G)\big]=\emm^{-2\theta n}\big|\cQ\backslash \cQ'\big|\,\frac{\sum_{(G,\sigma)\in \Lambda_{d,\beta}} \emm^{\beta \cH_G(\sigma)}}{\big|\cG(n,d)\big|}
\end{equation}
From the definition of $\big(\hat\G\big(\sferrod\big),\sferrod\big)$, cf. \eqref{eqGsigma},\eqref{eqsferrospara}, observe that 
\begin{equation*}
\frac{\sum_{(G,\sigma)\in \Lambda_{d,\beta}\backslash \cE} \emm^{\beta \cH_G(\sigma)}}{\sum_{(G,\sigma)\in \Lambda_{d,\beta}} \emm^{\beta \cH_G(\sigma)}}=\Pr\big[\big(\hat\G(\sferrod),\sferrod\big)\in \Lambda_{d,\beta}\backslash \cE\big]\leq\emm^{-\eta n}\leq \emm^{-10\theta n},
\end{equation*}
where the penultimate inequality follows from Proposition~\ref{thm:SWferro} and the last from the choice of $\theta$. Combining this with \eqref{eq:inter14445}, we obtain $\Pr[\G\in \cQ\backslash \cQ']=o(1)$. Since $\Pr[\G\in \cQ]=1-o(1)$ from Lemma~\ref{lem:phasebounds}, it follows that 
\[\Pr[\G\in \cQ\cap \cQ']\geq \Pr[\G\in \cQ]-\Pr[\G\in \cQ\backslash \cQ']\geq 1-o(1).\]
This concludes the proof for the metastability of the ferromagnetic phase $\tSferrod$ when $\beta>\uniq$.

A similar bottleneck-ratio argument shows that $\Sparad$ is a metastable state for $\beta<\KS$. The slow mixing of SW for $\beta\in (\uniq,\KS)$ follows from the metastability of $\tSferrod$ when $\beta\in (\uniq,\curie]$ and the metastability of $\Sparad$ when $\beta\in [\curie,\KS)$. In particular, let $S\in\{\tSferrod,\Sparad\}$ be such that $\norm{\mu\big(\nix\mid S\big)-\mu}\geq 1/2$, then Lemma~\ref{lem_bottleNeck}  gives that for $T=\emm^{\Omega(n)}$, it holds that $\norm{ \mu\big(\nix\mid S\big) P^{T}_{SW}-\mu}\geq 1/2-1/10$, yielding that the mixing time is $\emm^{\Omega(n)}$. 
\end{proof}

\section{Remaining Proofs for Swendsen-Wang}
To analyse the Swendsen-Wang dynamics on the $d$-regular random graph $\G$, we will need to consider the component structure  after performing edge percolation with probability $p\in (0,1)$. Key quantities we will be interested in are the size of the largest compoment, which will allow us to track whether we land in the paramagnetic or ferromagnetic phases, as well as the sum of squares of component sizes; the first will signify whether we land in the paramagnetic or ferromagnetic phases, and the second will allow us to track the random fluctuations caused by the colouring step of SW. Both of these ingredients have been worked out in detail for the mean-field case; here the random regular graph makes all the arguments more involved technically, even for a single iteration (recall that the reason it suffices to analyse a single iteration is because of the quiet planting idea of Sections~\ref{sec:quiet} and~\ref{sec:slowSW}).
\subsection{Percolation on random regular graphs}

For a graph $G$ and $p\in(0,1)$, we denote by $G_p$ the random graph obtained by keeping every edge of $G$ with probability $p$. Working in the configuration model, we will denote by $\G_{p}:=\G_{p}(n,d)$ the multigraph obtained by first choosing a random matching of the points in $[n]\times [d]$, then keeping each edge of the matching with probability $p$, and finally projecting the edges onto vertices in $[n]$. It will also be relevant to consider the  multigraph $\tG_p:=\tG_{p}(n,d)$ where in the second step we instead keep  a random subset of \emph{exactly} $m=[pdn/2]$ edges. To help differentiate between the two models, we will refer to $\G_{p}$ as the binomial-edge model, whereas to $\tG_p$ as the exact-edge model. Note that for an $n$-vertex multigraph $G$ of maximum degree $d$ with $m$ edges, the two models are related by 
\[\Pr\big[\G_{p}=G\mid E(\G_{p})=m\big]=\Pr[\tG_{\tp}=G], \mbox{ where } \tp=2m/nd.\]
see for example \cite[Lemma 3.1]{fountoulakis}. Based on this, it is standard to relate the two models for events that are monotone under edge inclusion.\footnote{A set of multigraphs $\cE$ is an increasing (resp. decreasing) property if for any $G=(V,E)\in \cE$, we have that $G'=(V,E')\in \cE$ for all $G'$ with $E'\subseteq E$ (resp. $E\subseteq E'$).}
\begin{lemma}\label{lem:translate12}
Let $d\geq 3$ be an integer and $p^*\in (0,1)$ be a constant. There exists a constant $c>0$ such that, for any constant $\delta\in (0,1)$, for any increasing property  $\cE$ and any decreasing property $\cF$ on multigraphs of maximum degree $d$, it holds that 
\begin{gather*}
\tfrac{1}{2}\pr[\tG_{p^*-\delta}\in \cE]\leq \Pr[\G_{p^*}\in \cE]\leq \pr[\tG_{p^*+\delta}\in \cE]\  +\emm^{-c\delta^2 n},\\
\tfrac{1}{2}\pr[\tG_{p^*+\delta}\in \cF]\leq \Pr[\G_{p^*}\in \cF]\leq \pr[\tG_{p^*-\delta}\in \cF]\  +\emm^{-c\delta^2 n}.
\end{gather*}
\end{lemma}
\begin{proof}
Let $\mathcal{A}$ be the event that $E(\G_{p^*})$ has $(p^*\pm\delta)dn/2$ edges. By standard Chernoff bounds we obtain that there exists a constant $c>0$ such that $\pr(\mathcal{A})\geq 1- \emm^{-c\delta^2 n}$.  Further, conditioned on $|E(\G_{p^*})|=p dn/2$ for some $p$, the graph $\G_{p^*}$ has the same distribution as $\tG_{p}$, and therefore, using the fact that $\cE$ is an  increasing property, we have that $\pr[\tG_{p^*+\delta}\in \cE]\geq \Pr[\tG_{p^*}\in \cE\mid \mathcal{A}]\geq \pr[\tG_{p^*-\delta}\in \cE]$, and the inequalities are reversed for $\cF$, yielding the lemma.
\end{proof}


It is a classical result \cite{Alon}  that  for percolation on random $d$-regular graphs there is a phase transition at $p=1/(d-1)$ with regards to the emergence of a giant component, see also \cite{MolloyReed,JansonLu,NachmiasPeres,highgirth}. To prove Propositions~\ref{thm:SWpara} and~\ref{thm:SWferro},  we will need  to control the sizes of the components in the strictly subcritical and supercritical regimes with probability bounds that are exponentially close to 1, which makes most of these results not directly applicable. 

For a graph $G$ and an integer $i\geq 1$, we denote by $C_i(G)$ the $i$-th largest component of $G$ (in terms of vertices);  $|C_i(G)|$ and $|E(C_i(G))|$ denote  the number of vertices and edges in $C_i(G)$. The following proposition gives the desired bound on the component sizes in the subcritical  regime.
\begin{proposition}\label{thm:subcritical}
Let $d\geq 3$ be an integer and $p_0<1/(d-1)$ be a positive constant. There exists constants $c,M>0$  such that the following holds for all integers $n$.  For any positive $p<p_0$, with probability at least $1-\emm^{-cn}$ over the choice of either $G\sim \G_p$ or $G\sim \tilde{\G}_p$, it holds that $\sum_{i\geq 1}|C_i(G)|^2\leq M n$.
\end{proposition}
\begin{proof}
The proof is fairly standard and actually holds for percolation on an arbitrary graph of maximum degree $d$. We argue initially for the binomial-edge case $G\sim \G_p$. Consider the process where we consider the vertices of $G$ in an arbitrary order, and we explore by breadth-first-search the components of those vertices that have not been discovered so far. Suppose that we have already explored the components $\mathcal{C}_1,\hdots, \mathcal{C}_k$ and we are exploring the component $\mathcal{C}_{k+1}$ starting from vertex $v$. Since the graph has maximum degree $d$,  the size of $\mathcal{C}_{k+1}$ is stochastically dominated above by a branching process where the root has offspring distribution  $\mathrm{Bin}(d,p_0)$ and every other vertex has $\mathrm{Bin}(d-1,p_0)$. Since $p_0<1/(d-1)$, the latter process is subcritical and therefore there exist constants $c',K>0$ (depending only on $d$ and $p_0$) such that for all $t>K$, it holds that 
\begin{equation}\label{eq:exponentialtails}
\Pr\big[|\mathcal{C}_{k+1}|>t\mid \mathcal{C}_1,\hdots, \mathcal{C}_k\big]\leq \emm^{-c' t}.
\end{equation}
 We have that $\sum_{i\geq 1}|C_i(G)|^2=\sum_{k\geq 1}|\mathcal{C}_k|^2\leq K^2 n+\sum_{k\geq 1}|\mathcal{C}_k|^2\mathbf{1}\{\mathcal{C}_k\geq K\}$. From~\eqref{eq:exponentialtails}, we have that the sum in the last expression is stochastically dominated by the sum of $n$ i.i.d. random variables with exponential tails, and therefore there exists constants $c,M'>0$, depending only on $p_0$, such that with probability $1-\emm^{-cn}$ the sum is bounded by $M'n$, yielding the result with $M=M'+K^2$. The exact-edge case $G\sim \tilde{\G}_p$ follows by applying Lemma~\ref{lem:translate12}, noting that the graph property $\sum_{i\geq 1}|C_i(G)|^2\leq M n$ is decreasing under edge-inclusion.
\end{proof}

The supercritical regime is more involved since we need to account for the giant component using large deviation bounds. While there is not an off-the-self result we can use, we can adapt a technique by Krivelevich,  Lubetzky and Sudakov \cite{highgirth} that was developed in a closely related setting (for high-girth expanders, refining previous results of Alon, Benjamini and Stacey \cite{Alon}).

 For $d\geq 3$ and $p\in (\frac{1}{d-1},1)$, let $\phi=\phi(p)\in (0,1)$ be the probability that a branching process with offspring distribution $\mathrm{Bin}(d-1,p)$ dies out, i.e., $\phi(p)\in (0,1)$ is  the (unique) solution of
\begin{equation}\label{eq:kappachipsi}
\phi =(p \phi+1-p)^{d-1}, \mbox{ and define } \chi=\chi(p), \psi=\psi(p) \mbox{ from } \chi=1-(p \phi+1-p)^{d},\quad \psi=\tfrac{1}{2}dp(1-\phi^2).
\end{equation}
In Appendix~\ref{sec:giantperc}, we show the following adapting the argument from \cite{highgirth}.
\newcommand{\statelemgiantperc}{Let $d\geq 3$ be an integer, $p\in (\frac{1}{d-1},1)$ be a real, and $\chi=\chi(p),\psi=\psi(p)$ be as in \eqref{eq:kappachipsi}. Then, for any $\delta>0$, with probability $1-\emm^{-\Omega(n)}$ over the choice of either $G\sim \G_p$ or $G\sim \tilde{\G}_p$, it holds that
\begin{equation*}
|C_1(G)|=(\chi\pm \delta) n, \quad |E(C_1(G))|=(\psi\pm \delta) n.
\end{equation*}}
\begin{lemma}\label{lem:giantperc}
\statelemgiantperc
\end{lemma}
With this and a bit of algebra, we can derive the analogue of Proposition~\ref{thm:subcritical} in the supercritical regime.
\begin{proposition}\label{cor:supercritical}
Let $d\geq 3$ be an integer. Consider arbitrary  $p_0\in (\tfrac{1}{d-1},1)$ and let $\chi_0=\chi(p_0)$ be as in \eqref{eq:kappachipsi}. Then, for all $\delta>0$, there exist $\eps,c,M>0$, such that the following holds.  For all sufficiently large integers $n$ and any $p=p_0\pm \eps$,   with probability at least $1-\emm^{-cn}$ over the choice of either $G\sim \G_p$ or $G\sim \tilde{\G}_p$, it holds that $|C_1(G)|=(\chi_0 \pm \delta) n$ and  $\sum_{i\geq 2}|C_i(G)|^2\leq M n$.
\end{proposition}
To prove Proposition~\ref{cor:supercritical}, the following inequality between $\chi$ and $\psi$ will be useful; it will allow us to conclude that once we remove the giant component, the remaining components are in the subcritical regime.
\begin{lemma}\label{lem:subcritineq}
Let $d\geq 3$ be an integer and $p\in (\tfrac{1}{d-1},1)$. Then, $\frac{2(\tfrac{1}{2}dp-\psi)}{d(1-\chi)}< \tfrac{1}{d-1}$.
\end{lemma}
\begin{proof}
Using \eqref{eq:kappachipsi}, we have 
\[\tfrac{d(1-\chi)}{d-1}-2(\tfrac{1}{2}dp-\psi)=\tfrac{d}{d-1}(p \phi+1-p)^{d}-dp\phi(p \phi+1-p)^{d-1}=\tfrac{d}{d-1}(p \phi+1-p)^{d-1}(1-p-(d-2)p\phi),\]
so it suffices to show that $1-p-(d-2)p\phi>0$. Let $g(y)=y -(p y+1-p)^{d-1}$ and note that $g(\phi)=0$. Then, we have that $g(0)<0$ and $g(1)=0$. Moreover, $g'(y)=1-(d-1)p(py+1-p)^{d-2}$  and hence $g'(1)<0$. It follows that $g(y)>0$ for $y\uparrow 1$, and therefore there is $y\in (0,1)$ such that $g(y)=0$. Note that  $g$ is strictly concave and therefore cannot have three zeros in the interval $(0,1]$, so $y=\phi$, and  therefore $g'(\phi)> 0$.  It remains to observe that $g'(\phi)=\tfrac{1-p-(d-2)p\phi}{p \phi+1-p}$, from where the desired inequality follows.
\end{proof}
\begin{proof}[Proof of Proposition~\ref{cor:supercritical}]
Let $\psi_0=\psi(p_0)$ and consider an arbitrarily small $\delta>0$. Since $\chi(p)$ and $\psi(p)$ are continuous functions of $p$ in the interval $(\tfrac{1}{d-1},1)$, we can pick $\eps>0$ so that, for all $p=p_0\pm \eps$ it holds that  $d|p-p_0|, |\chi(p)-\chi_0|,|\psi(p)-\psi_0|\leq \delta/10$ and, by Lemma~\ref{lem:subcritineq}, $\frac{2(\tfrac{1}{2}dp-\psi)+4\delta}{d(1-\chi)-\delta}<\tfrac{1}{d-1}-\delta$. Consider now an arbitrary $p=p_0\pm \eps$ and consider random $G$ sampled from either of the distributions  $\G_p$ or $\tG_p$. Using the monotonicity of the events $\{|C_1(G)|\geq t\},\{|E(C_1(G))|\geq t\}$, we obtain from Lemmas~\ref{lem:translate12} and ~\ref{lem:giantperc} (as well as a standard Chernoff bound for the number of edges in $G$) that there exists a constant $c'>0$, depending only on $d, p_0,\eps$ (but not on $p$), such that with probability at least $1-\emm^{-c'n}$ over the choice of $G$ it holds that $|E(G)|=\tfrac{1}{2}{dpn}\pm \delta n$, $|C_1(G)|=(\chi_0 \pm \delta) n$, and $|E(C_1(G))|=(\psi_0 \pm \delta) n$. Let $\cE$ denote this event. 

Note that conditioned on $|C_1(G)|, |E(C_1(G))|$ and $|E(G)|$, the remaining components of $G$ are distributed according to those in the exact-edge model $\tG_{\tp}(\tilde{n},d)$ with $\tn=n-|C_1(G)|$ and $\tilde{p}=\tfrac{2}{d\tn}(|E(G)|-|E(C_1(G))|)$, conditioned on the event $\cF$ that all components have size less than $|C_1(G)|$. Hence, conditioned on $\cE$, we have that $\tp\leq \frac{2(\frac{1}{2}dpn-\psi n)+4\delta n}{2(n-\chi n)-\delta n}<\tfrac{1}{d-1}-\delta$ where the last inequality follows from the choice of $\eps$, i.e., $\tG_{\tp}(\tilde{n},d)$ is in the subcritical regime. Therefore, the probability of $\cF$ is $1-\emm^{-\Omega(n)}$ and hence the conditioning on $\cF$ when considering $\tG_{\tp}(\tilde{n},d)$ can safely be ignored. From Proposition~\ref{thm:subcritical}, we have that there exist constants $M,c''>0$,  depending only on $d$ and $p_0$, so that with probability at least $1-\emm^{-c''n}$ over the choice of $G'\sim \tG_{\tp}(\tilde{n},d)$,  it holds that $\sum_{i\geq 1}|C_i(G')|^2\leq M \tilde{n}$. Therefore, we have $\sum_{i\geq 2}|C_i(G)|^2\leq M n$.
\end{proof}

\subsection{Percolation in the planted model}
Recall the edge-empirical distributions $\rho_{G,\sigma}$, $\rhopara$, $\rhoferro$, cf. \eqref{eqLemma_Jean}. The following lemma will allow us to deduce the regime (subcritical or supercritical) that dictates the percolation step of SW when we start from the paramagnetic and ferromagnetic phases.
\begin{lemma}\label{lem:controlperc}
For $\beta<\KS$, any colour $s\in [q]$ in the paramagnetic phase satisfies $(1-\emm^{-\beta})\frac{\rhopara(s,s)}{\nupara(s)}<\tfrac{1}{d-1}$.  For $\beta>\uniq$, any colour $s\in [q]$ in the ferromagnetic phase satisfies $(1-\emm^{-\beta})\frac{\rhoferro(s,s)}{\nuferro(s)}=\tfrac{(\emm^{\beta}-1)\muferro(s)}{1+(\emm^{\beta}-1)\muferro(s)}$; this is larger than $\tfrac{1}{d-1}$ for the colour $s=1$, and less than $\tfrac{1}{d-1}$ for all the other $q-1$ colours. 
\end{lemma}
\begin{proof}
For the paramagnetic phase and any  colour $s\in [q]$, it follows from \eqref{eqLemma_Jean}  that
\[\nupara(s)=\tfrac{1}{q}, \qquad \rhopara(s,s)=\tfrac{\emm^{\beta}}{q\emm^{\beta}+(q^2-q)},\]
so $(1-\emm^{-\beta})\frac{\rhopara(s,s)}{\nupara(s)}<\tfrac{1}{d-1}$ is equivalent to $(1-\emm^{-\beta})\frac{\emm^{\beta}}{\emm^{\beta}+q-1}<\tfrac{1}{d-1}$ which is true iff $\beta<\KS$, since $\KS=\log(1+\tfrac{q}{d-2})$.

For the ferromagnetic phase, recall from Section~\ref{sec:states} that $x=\muferro(1)$ is the largest number in the interval $(1/q,1)$ that satisfies 
\begin{equation}\label{eq:xerr}
x=\frac{(1+(\emm^\beta-1)x)^{d-1}}{(1+(\emm^\beta-1)x)^{d-1}+(q-1)\big(1+(\emm^\beta-1)\tfrac{1-x}{q-1}\big)^{d-1}}.
\end{equation}
Let $t=\tfrac{1+(\emm^\beta-1)x}{1+(\emm^\beta-1)\tfrac{1-x}{q-1}}$ and note that $t>1$ since $x>1/q$ and $\beta>0$. Moreover, \eqref{eq:xerr} can be written as $x=\frac{t^{d-1}}{t^{d-1}+(q-1)}$, and hence $t^{d-1}=\frac{(q-1)x}{1-x}$. Then, it follows from \eqref{eqLemma_Jean} that for colour $s=1$ we have
\begin{equation}\label{eq:nuferro}
\nuferro(1)=\frac{t^d}{t^d+(q-1)}=\frac{t x}{tx+1-x}, \qquad \rhoferro(1,1)=\frac{\emm^\beta x^2}{1+(\emm^\beta-1)\big(x^2+\tfrac{(1-x)^2}{q-1}\big)}=\frac{\emm^\beta  t x^2}{(tx+1-x)\big(1+(\emm^\beta-1)x\big)},
\end{equation}
whereas for colours $s\neq 1$ we have \[\nuferro(s)=\frac{1}{t^d+(q-1)}=\frac{\tfrac{1-x}{q-1}}{tx+1-x}, \qquad \rhoferro(s,s)=\frac{\emm^\beta \big(\tfrac{1-x}{q-1}\big)^2}{1+(\emm^\beta-1)(x^2+\tfrac{(1-x)^2}{q-1})}=\frac{\emm^\beta  t \big(\tfrac{1-x}{q-1}\big)^2}{(tx+1-x)\big(1+(\emm^\beta-1)x\big)}.\]
Using these expressions, it is a matter of few manipulations to verify that $(1-\emm^{-\beta})\frac{\rhoferro(s,s)}{\nuferro(s)}=\tfrac{(\emm^{\beta}-1)\muferro(s)}{1+(\emm^{\beta}-1)\muferro(s)}$ for all colours $s\in [q]$.

Using this, for $s=1$, we have that the inequality $(1-\emm^{-\beta})\frac{\rhoferro(1,1)}{\nuferro(1)}>\tfrac{1}{d-1}$ is equivalent to $(\emm^\beta-1)x>\tfrac{1}{d-2}$. Plugging $x=\frac{t^{d-1}}{t^{d-1}+(q-1)}$ into $t=\tfrac{1+(\emm^\beta-1)x}{1+(\emm^\beta-1)\tfrac{1-x}{q-1}}$ and solving for $(\emm^\beta-1)$ yields that $\emm^\beta-1=\frac{(t-1)(t^{d-1}+q-1)}{t^{d-1}-t}$. Therefore the desired inequality becomes 
\[\frac{(t-1)t^{d-1}}{t^{d-1}-t}>\tfrac{1}{d-2}, \mbox{ or equivalently } (d-2)t^{d-1}-(d-1)t^{d-2}+1>0,\]
which is true for any $t>1$. For a colour $s\neq 1$, the inequality $(1-\emm^{-\beta})\frac{\rhoferro(s,s)}{\nuferro(s)}<\tfrac{1}{d-1}$ can be proved analogously. We have in particular  the equivalent inequality $(\emm^\beta-1)\tfrac{1-x}{q-1}<\tfrac{1}{d-2}$, which further reduces to  $\frac{t-1}{t^{d-1}-t}<\tfrac{1}{d-2}$; the latter again holds for any $t>1$.
\end{proof}

\subsection{Tracking one step of SW - Proof of Propositions~\ref{thm:SWpara} and~\ref{thm:SWferro}}
\begin{PropSWpara}
\statepropSWpara
\end{PropSWpara}
\begin{proof}
Let $\eps>0$ be a sufficiently small constant so that by Lemma~\ref{Lemma_spara}, for any constant $\delta>0$, with probability $1-\emm^{-\Omega(n)}$ over the choice of $(G,\sigma)\sim \big(\hat\G(\sparad),\sparad\big)$,  we have 
\begin{equation}\label{eq:numparaeps}
\norm{\nu^{\sigma}-\nupara}\leq \delta\mbox{  and  }\norm{\rho^{G,\sigma}-\rhopara}\leq \delta.
\end{equation} 
Let $\eps'$ be an arbitrary constant such that $\eps'>\eps$. We will show that there exists a constant $\eta>0$ such that for arbitrary $\nu$ and $\rho\in\cR(\nu)$ satisfying $\norm{\nu-\nupara}\leq \delta$ and $\norm{\rho-\rhopara}\leq \delta$,  for $(G,\sigma)\sim \big(\hat\G(\sparad),\sparad\big)$, it holds that
\begin{equation}\label{eq:PSWpara}
\Pr\Big[P^{G}_{SW}\big(\sigma\to\Sparadp\big)\geq  1-e^{-\eta n}\,\big|\, \nu^{\sigma}=\nu, \rho^{G,\sigma}=\rho\Big]\geq 1-\emm^{-\eta n}
\end{equation}
and therefore the conclusion follows by aggregating over $\nu$ and $\rho$, using the law of total probability and the probability bound for \eqref{eq:numparaeps}.

Choose $(G,\sigma)\sim \big(\hat\G(\sparad),\sparad\big)$ conditioned on $\nu^{\sigma}=\nu$ and $\rho^{G,\sigma}=\rho$. Observe that $\hat\G(\sigma)$ is a uniformly random graph conditioned on the sizes of the vertex/edge classes prescribed by $\nu,\rho$. For $i\geq1$, let $C_i(G_{\sigma,SW})$ be the components of $G$ (in decreasing order of size) starting from the configuration $\sigma$ after the percolation step of the SW dynamics with parameter $p=1-\emm^{-\beta}$, when starting from the configuration $\sigma$. We will show that there exists a constant $M>0$ such that
\begin{equation}\label{eq:sumofsquarespara}
\Pr\bigg[\sum_{i\geq 1}|C_i(G_{\sigma,SW})|^2\leq Mn\,\Big|\, \nu^{\sigma}=\nu, \rho^{G,\sigma}=\rho\bigg]\geq 1-\emm^{-\Omega(n)}.
\end{equation}
Assuming this for the moment, for a colour $s\in [q]$, let $N_s$ be the number of vertices with colour $s\in [q]$ after the recoloring step of SW. Note that the expectation of $N_s$ is $n/q$, and whenever the event in \eqref{eq:sumofsquarespara} holds, by Azuma's inequality we obtain that $\frac{1}{n}N_s$ is within an additive $\eps'$ from its expectation with probability $1-\emm^{-\Omega(n)}$. By a union bound over the $q$ colours, we obtain \eqref{eq:PSWpara}. 

For a colour $s\in [q]$, let $G(\sigma^{-1}(s))$ be the induced graph on $\sigma^{-1}(s)$, and note that since $G$ is uniformly random conditioned on $\nu$ and $\rho$, $G(\sigma^{-1}(s))$ has the same distribution as the exact-edge model $H(s)\sim\tilde{\G}_{\tr(s)}(\tilde{n}(s),d)$ where $\tilde{n}(s)=n\nu(s)$ and $\tr(s)=\frac{\rho(s,s)}{\nu(s)}$.  Percolation on this graph with parameter $p$ is therefore closely related to the binomial-edge  model $\G_{r(s)}(\tn(s),d)$ with $r(s)=p\tr(s)$. More precisely, note that for all sufficiently small $\delta>0$, Lemma~\ref{lem:controlperc}  guarantees that the percolation parameter $r(s)$ is bounded by a constant strictly less than $1/(d-1)$, so by Theorem~\ref{thm:subcritical} there exists a constant $M>0$ such that 
\begin{equation}\label{eq:eerf4455}
\Pr\Big[\sum_{i\geq 1}|C_i(\G_{r(s)})|^2\leq M\tn(s)\Big]\geq 1-\emm^{-\Omega(\tr{n}(s))}\geq 1-\emm^{-\Omega(n)}.
\end{equation}
Note that, for any $p\in (0,1)$, the property $\big\{G\,:\, \Pr\big[\sum_{i\geq 1}|C_i(G_p)|^2\leq Mn\big]\geq 1-\emm^{-\Omega(n)}\big\}$ is a  decreasing graph property, i.e., if $G$ is a subgraph of $G'$, we can couple the random graphs $G_p$ and $G_p'$ so that $\sum_{i\geq 1}|C_i(G_p)|^2\leq \sum_{i\geq 1}|C_i(G_p')|^2$. Viewing the event in \eqref{eq:eerf4455} as a property of the binomial-edge model $\G_{\tr(s)}(\tn(s),d)$, it follows from Lemma~\ref{lem:translate12} that with probability $1-\emm^{-\Omega(n)}$ over the choice of the exact-edge model $H(s)\sim\tilde{\G}_{\tr(s)}(\tilde{n}(s),d)$ it holds that
\[\Pr\Big[\sum_{i\geq 1}|C_i(H_p(s))|^2]\leq M\tn(s)\Big]\geq 1-\emm^{-\Omega(n)}.\]
Applying this for colours $s=1,\hdots,q$ and $H(s)=G(\sigma^{-1}(s))$, we obtain by the union bound that with probability $1-\emm^{-\Omega(n)}$ over the choice of $(G,\sigma)\sim \big(\hat\G(\sparad),\sparad\big)$ conditioned on $\nu^{\sigma}=\nu$ and  $\rho^{G,\sigma}=\rho$, the components of $G$ after the percolation step of SW satisfy \eqref{eq:sumofsquarespara}, as claimed, therefore finishing the proof.
\end{proof}

\begin{PropSWferro}
\statepropSWpara
\end{PropSWferro}
\begin{proof}[Proof of Proposition~\ref{thm:SWferro}]
The first part of the proof is analogous to that of Theorem~\ref{thm:SWpara}. Let $\eps>0$ be a sufficiently small constant, so that by Lemma~\ref{Lemma_sferro}, for any constant $\delta>0$, with probability $1-\emm^{-\Omega(n)}$ over the choice of $(G,\sigma)\sim \big(\hat\G(\sferrod),\sferrod\big)$,  we have 
\begin{equation}\label{eq:numferroeps}
\norm{\nu^{\sigma}-\nuferro}\leq \delta\mbox{  and  }\norm{\rho^{G,\sigma}-\rhoferro}\leq \delta.
\end{equation}  We will show that there exists a constant $\eta>0$ such for arbitrary $\nu$ and $\rho\in\cR(\nu)$ satisfying $\norm{\nu-\nuferro}\leq \delta$ and $\norm{\rho-\rhoferro}\leq \delta$, for  $(G,\sigma)\sim \big(\hat\G(\sferrod),\sferrod\big)$ it holds that
\begin{equation}\label{eq:PSWferro}
\Pr\Big[P^{G}_{SW}\big(\sigma\to\tSferrodp\big)\geq  1-e^{-\eta n}\,\big|\, \nu^{\sigma}=\nu, \rho^{G,\sigma}=\rho\Big]\geq 1-\emm^{-\eta n}
\end{equation}
and therefore the conclusion follows by aggregating over $\nu$ and $\rho$.

Choose $(G,\sigma)\sim \big(\hat\G(\sferrod),\sferrod\big)$ conditioned on $\nu^\sigma=\nu$ and $\rho^{G,\sigma}=\rho$, and observe once again that $\hat\G(\sigma)$ is uniformly random conditioned on $\nu,\rho$. For $i\geq1$, let $C_i(G_{\sigma,SW})$ be the components of $G$ (in decreasing order of size) starting from the configuration $\sigma$ after the percolation step of the SW dynamics with parameter $p=1-\emm^{-\beta}$, when starting from the configuration $\sigma$. We will show that there exists a constant $M>0$ such that
\begin{equation}\label{eq:sumofsquaresferro}
\Pr\Big[C_1(G_{\sigma,SW})= n\big(1-\tfrac{q(1-\nuferro(1))}{q-1}\big)\pm \epsilon'n,\ \sum_{i\geq 2}|C_i(G_{\sigma,SW})|^2\leq Mn\,\Big|\, \nu^{\sigma}=\nu, \rho^{G,\sigma}=\rho\Big]\geq 1-\emm^{-\Omega(n)}.
\end{equation}
We first complete the proof of the theorem assuming this for the moment, and return to the proof of \eqref{eq:sumofsquaresferro} later. In particular, assume w.l.o.g. that $C_1(G_{\sigma,SW})$ gets colour 1. For a colour $s\in [q]$, let $N_s$ be the number of vertices outside $C_1(G_{\sigma,SW})$ that get colour $s\in [q]$ after the recoloring step of SW. Note that in the final configuration after the recolouring step, the number of vertices with colour $s\in [q]$  is $N_s+\mathbf{1}\{s=1\}|C_1(G_{\sigma,SW})|$. Now, the expectation of $N_s$ is $\tfrac{n-|C_1(G_{\sigma,SW})|}{q}$, and whenever the event in \eqref{eq:sumofsquarespara} holds, by Azuma's inequality we obtain that $\frac{1}{n}N_s$ is within an additive $\eps'$ from its expectation with probability $1-\emm^{-\Omega(n)}$. Therefore, by a union bound over the $q$ colours, the Potts configuration obtained after one step of SW belongs to $\tSferrodp$ with probability $1-\emm^{-\Omega(n)}$, which establishes the claim in \eqref{eq:PSWferro}. 

It remains to prove \eqref{eq:sumofsquaresferro}. As in the proof of Proposition~\ref{thm:SWpara}, for a colour $s\in [q]$, let $G(\sigma^{-1}(s))$ be the induced graph on $\sigma^{-1}(s)$, and note that $G(\sigma^{-1}(s))$ has the same distribution as the exact-edge model $H(s)\sim\tilde{\G}_{\tr(s)}(\tilde{n}(s),d)$ where $\tn(s)=n\nu(s)$ and $\tr(s)=\frac{\rho(s,s)}{\nu(s)}$.  By considering again the binomial-edge model $\G_{r(s)}(\tn(s),d)$ with $r(s)=p\tr(s)$, and using the inequalities in Lemma~\ref{lem:controlperc} for the ferromagnetic phase, we obtain that for all colours $s\neq 1$ the parameter $r(s)$ is bounded by a constant strictly less than $\tfrac{1}{d-1}$ and hence the model is in the subcritical regime. In fact, by the same line of arguments as in Theorem~\ref{thm:SWpara},  we therefore have that there exists a constant $M_0>0$ (depending only on $d,\beta$ but not on $\nu$ or $\rho$) such that, for all colours $s\neq 1$ with probability $1-\emm^{-\Omega(n)}$ over the choice of $H(s)\sim\tG_{\tr(s)}(\tilde{n}(s),d)$, it holds that
\begin{equation}\label{eq:1stcolour}
\Pr\Big[\sum_{i\geq 1}|C_i(H_p(s))|^2\leq M_0\tn(s)\Big]\geq 1-\emm^{-\Omega(n)}
\end{equation}
By contrast, for $s=1$, the binomial-edge model $\G_{r(s)}(\tn(s),d)$ is in the supercritical regime since $r(s)=\rferro\pm \eps$ where $\rferro=(1-\emm^{-\beta})\frac{\rhoferro(1,1)}{\nuferro(1)}=\tfrac{(\emm^{\beta}-1)\muferro(1)}{1+(\emm^{\beta}-1)\muferro(1)}$ is a constant larger than $\tfrac{1}{d-1}$ (by Lemma~\ref{lem:controlperc}). Let $\chiferro=\chi(\rferro)$ be as in~\eqref{eq:kappachipsi}, so by Proposition~\ref{cor:supercritical} there exists a constant $M_1>0$ such that 
\begin{equation}\label{eq:2ndcolour}
\Pr\Big[\big|C_1(\G_{r(s)})\big|= \tn(s)(\chiferro\pm\tfrac{\epsilon'}{2})\Big],\ \Pr\bigg[\sum_{i\geq 2}|C_i(\G_{r(s)})|^2\leq M_1\tn(1)\bigg]\geq 1-\emm^{-\Omega(n)}.
\end{equation}
 We will shortly show that 
\begin{equation}\label{eq:algebra}
1-\frac{q(1-\nuferro(1))}{q-1}=\chiferro\nuferro(1) \mbox{ or equivalently } \chiferro= \frac{q\nuferro(1)-1}{(q-1)\nuferro(1)}.
\end{equation}
Assuming this for now, note that since $|C_1(G)|$ and $\Pr\big[\sum_{i\geq 2}|C_i(G_p)|^2\big]$ are monotone under edge-inclusion, we can again use Lemma~\ref{lem:translate12} to go back to the percolation model for the colour $s= 1$. So, we conclude that with probability $1-\emm^{-\Omega(n)}$ over the choice of $H(s)\sim\tG_{\tr(s)}(\tilde{n}(s),d)$, it holds that
\[\Pr\bigg[\big|C_1(H_p(s))\big|= \tn(s)(\chiferro\pm\epsilon'), \quad \sum_{i\geq 1}|C_i(H_p(s))|^2\leq M_1\tn(s)\bigg]\geq 1-\emm^{-\Omega(n)}.\] 
Combining \eqref{eq:1stcolour} and \eqref{eq:2ndcolour} with a union bound over the $q$ colours, we obtain \eqref{eq:sumofsquaresferro} with $M=\max{M_0,M_1}$.

It only remains to prove \eqref{eq:algebra}. Recall from~\eqref{eq:nuferro} that $\nuferro(1)=\tfrac{t^d}{t^d+(q-1)}$ where $t=\tfrac{1+(\emm^\beta-1)x}{1+(\emm^\beta-1)\tfrac{1-x}{q-1}}$ and $x=\muferro(1)$. So, $\chiferro= \frac{q\nuferro(1)-1}{(q-1)\nuferro(1)}$ is equivalent to showing that 
\begin{equation}\label{eq:algebra1}
\chiferro= 1-(1/t)^d.
\end{equation}
Now, recall from \eqref{eq:kappachipsi} that $\chiferro=1-\big(1-\rferro+\rferro\phiferro\big)^{d}$, where $\phiferro=\phi(\rferro)$. So \eqref{eq:algebra1} reduces to showing that
\begin{equation}\label{eq:algebra2}
1/t=1-\rferro+\rferro\phiferro, \mbox{which using $t=\tfrac{1+(\emm^\beta-1)x}{1+(\emm^\beta-1)\tfrac{1-x}{q-1}}$ and $\rferro=\tfrac{(\emm^{\beta}-1)x}{1+(\emm^{\beta}-1)x}$  is  equivalent to }\phiferro=\tfrac{1-x}{(q-1)x}.
\end{equation}
From \eqref{eq:kappachipsi}, $y=\phiferro$ is the unique solution in $(0,1)$ of the equation
\begin{equation}\label{eq:algebra3}
y=\big(1-\rferro+\rferro y\big)^{d-1},
\end{equation}
and note that $\tfrac{1-x}{(q-1)x}\in(0,1)$ since $x>1/q$. So, to prove the equality $\phiferro=\tfrac{1-x}{(q-1)x}$ in  \eqref{eq:algebra2}, it suffices to show that setting $y=\tfrac{1-x}{(q-1)x}$ satisfies \eqref{eq:algebra3}. This follows from the fact that $x=\muferro(1)$ satisfies the Belief propagation equations; in particular, from \eqref{eq:xerr} we have 
\[x=\frac{(1+(\emm^\beta-1)x)^{d-1}}{(1+(\emm^\beta-1)x)^{d-1}+(q-1)\big(1+(\emm^\beta-1)\tfrac{1-x}{q-1}\big)^{d-1}},\]
from which it follows that  $y=\tfrac{1-x}{(q-1)x}=\Big(\frac{1+(\emm^\beta-1)x)}{1+(\emm^\beta-1)\tfrac{1-x}{q-1}}\Big)^{d-1}=\big(1-\rferro+\rferro y\big)^{d-1}$. This finishes the proof of \eqref{eq:algebra} and therefore the proof of Proposition~\ref{thm:SWferro}.
\end{proof}

\section{Declarations}
A. Coja-Oghlan supported by DFG CO 646/3 and 646/4. J.B. Ravelomanana supported by DFG CO 646/4. D. \v{S}tefankovi\v{c} supported by NSF CCF-2007287. E. Vigoda supported by NSF CCF-2007022.

The authors have no competing interests to declare that are relevant to the content of this article.

\printbibliography

\begin{appendix}

\section{Proof of \Lem~\ref{lem_bottleNeck}}\label{sec_lem_bottleNeck}
\begin{Lembottle}
\statelembottle
\end{Lembottle}
\begin{proof}
We adapt the argument from \cite[proof of \Thm~7.3]{LPW}.
For $\sigma, \tau \in [q]^{V}$ and $A, B  \subseteq [q]^{V}$ let
\begin{equation*}
Q(\sigma,\tau)= \mu_{G}(\sigma)P(\sigma,\tau)  \quad \mbox{ and } \quad Q(A,B)= \sum_{\sigma \in A, \tau \in B} Q(\sigma, \tau).
\end{equation*}
Moreover, for a set $S \subseteq [q]^{V}$, let $\mu_{G,\beta,S} = \mu_{G,\beta}\bc{ \cdot \vert S}$.
We have \begin{equation}\label{eq_bottle1}
 	\mu_{G}(S) \norm{ \mu_{G,\beta,S} P - \mu_{G,\beta,S} }_{TV} = \mu_{G}(S) \sum_{\substack{ \sigma \in [q]^{V_n} \\ \mu_{G,\beta,S}P(\sigma)\geq \mu_{G,\beta,S}(\sigma) } }   \bc{\mu_{G,\beta,S} P (\sigma) - \mu_{G,\beta,S} (\sigma)}.
 \end{equation}
 Now, by definition, $\mu_{G,\beta,S}(\tau) = {\mu_{G} \bc{ \cbc{ \tau} \cap S } }/{\mu_{G} \bc{S} }$ so $\mu_{G,\beta,S}(\tau)=0$ if $\tau \notin S$ and $ \mu_{G,\beta,S}(\tau)={\mu_{G}(\tau)}/{\mu_{G}\bc{S} } $ otherwise. Hence,
 \begin{align} 
 \mu_{G}(S) \mu_{G,\beta,S}P(\sigma) = \sum_{ \tau \in [q]^{V}}  \mu_{G}(S) \mu_{G,\beta,S}(\tau) P( \tau, \sigma) = \sum_{\tau \in S}  \mu_{G}(\tau) P( \tau, \sigma) \leq  \sum_{\tau \in [q]^{V}}  \mu_{G}(\tau) P( \tau, \sigma)= \mu_{G}(\sigma) \label{eq_bottle2}
 \end{align}
 where the last equality in \eqref{eq_bottle2} holds because $\mu_G$ is the stationary distribution. Next, dividing \eqref{eq_bottle2} through by $\mu_{G}(S) $ and using the fact
 that $\mu_{G,\beta,S}(\tau)={\mu_{G}(\tau)}/{\mu_{G}\bc{S} } $ for $\tau \in S$, we have
 \begin{equation}\label{eq_bottle3}
\mu_{G,\beta,S}P(\tau)\leq  \mu_{G,\beta,S}(\tau) \qquad \mbox{for } \tau \in S.
 \end{equation}
 Furthermore, since $\mu_{G,\beta,S}(\tau)=0$ for $\tau \in S^{c}$,
 \begin{equation}\label{eq_bottle4}
 \mu_{G,\beta,S}P(\tau)\geq  \mu_{G,\beta,S}(\tau)=0 \qquad \mbox{for } \tau \in S^c.
 \end{equation}
Combining \eqref{eq_bottle3}, \eqref{eq_bottle4} and again the fact that $\mu_{G,\beta,S} (\sigma) =0$ for $\sigma \in S^c$ we see that  Equation  \eqref{eq_bottle1} becomes 
\begin{equation} \label{eq_bottle5}
	\mu_{G}(S) \norm{ \mu_{G,\beta,S} P - \mu_{G,\beta,S} }_{TV} =  \sum_{ \sigma \in S^c }  \mu_{G}(S)  \mu_{G,\beta,S} P (\sigma) .
\end{equation}
Once more, since $\mu_{G,\beta,S}(\tau)=0$ if $\tau \in S^c$ and $\mu_{G,\beta,S}(\tau)= \mu_{G}(\tau)/ \mu_{G}(S)$ if $\tau \in S$ we have
\begin{equation}\label{eq_bottle6}
\sum_{ \sigma \in S^c }	\mu_{G}(S)  \mu_{G,\beta,S} P (\sigma) = \sum_{ \sigma \in S^c }  \sum_{ \tau \in S }	\mu_{G}(S)  \mu_{G,\beta,S}(\tau) P ( \tau, \sigma )= \sum_{ \sigma \in S^c }  \sum_{ \tau \in S }  \mu_{G}(\tau) P ( \tau, \sigma)=Q(S,S^c)
\end{equation}
Combining \eqref{eq_bottle5} and \eqref{eq_bottle6}, we obtain
\begin{equation*}
	\mu_{G}(S) \norm{ \mu_{G,\beta,S} P - \mu_{G,\beta,S} }_{TV} = Q(S,S^c), \mbox{ and hence } 
\norm{ \mu_{G,\beta,S} P - \mu_{G,\beta,S} }_{TV} = \Phi(S,S^c).
\end{equation*}
In addition,  for any $u \geq 0$, it is easy to see that  we have
$$\norm{ \mu_{G,\beta,S} P^{u+1} - \mu_{G,\beta,S} P^{u} }_{TV} \leq \norm{ \mu_{G,\beta,S} P - \mu_{G,\beta,S} }_{TV}= \Phi(S,S^c).$$
Therefore, the result follows using the triangle inequality on the telescoping sum \[\mu_{G,\beta,S} P^t - \mu_{G,\beta,S}= \sum_{u=0}^{t-1}\mu_{G,\beta,S} P^{u+1} - \mu_{G,\beta,S} P^{u}.\qedhere\]
\end{proof}

\section{Proof of Lemma~\ref{lem:giantperc}}\label{sec:giantperc}
The proof follows closely the approach in \cite{highgirth} that was carried out for high-girth expanders. While the random regular graph is an expander itself, it contains a few small cycles and we only need to adapt the argument in order to account for their presence.
\begin{Lemmagiantperc}
\statelemgiantperc
\end{Lemmagiantperc}
\begin{proof}
Let $\delta>0$ be an arbitrarily small constant, and set $\eta=\delta/(100dp)$. It suffices to prove the result for the binomial-edge model $\G_p$, the result for the exact-edge model $\tG_p$ follows from Lemma~\ref{lem:translate12} since $|C_1(G)|$ and $|E(C_1(G))|$ are monotone under edge-inclusion. 

Let $\eps\in (0,p)$ be an arbitrarily small constant to be chosen later, and let   $\hat{p}:=\frac{p-\eps}{1-\eps}$; note  that  $\eps=\frac{p-\hat{p}}{1-\hat{p}}$.  We can think of the construction of $\G_p$ into the following steps: (i) sample a random $d$-regular graph $G=(V,E)\sim \G$ from the pairing model, (ii) keep each of the edges in $E$ independently with probability $\hat{p}$, to obtain the edge set $\hat{E}$, (iii) keep each of the edges in $E$ independently with probability $\eps>0$, to obtain the edge set $E_\eps$, (iv) the final graph has vertex set $V$ and edge set $\hat{E}\cup E_{\eps}$.

For a large integer $R>0$ to be chosen later, let $\phi_R$ be the probability that a branching process with offspring distribution $\mathrm{Bin}(d-1,\hat{p})$ has died out after $R$ generations, and let $\chi_R=1-(1-\hat{p}+\hat{p}\phi_R)^d$, $\psi_R=\tfrac{1}{2}d\hat{p}\big(1-\phi_R^2\big)$. Then,  by choosing $\eps>0$ sufficiently small, for all sufficiently large $R$ we have that $|\chi_R-\chi|\leq \eta$, $|\psi_R-\psi|\leq \eta$.

It is a well-known fact that  the random regular graph $G=(V,E)\sim \G$, i.e., the graph after step (i), is an expander and the local neighbourhoods of all but a small fraction of the vertices are trees. More precisely, there is a constant $\zeta>0$ such that for any integer $R>0$ the following hold with probability $1-\emm^{-\Omega(n)}$:
\begin{enumerate}
\item\label{it:neigh} the $(2R)$-neighbourhoods of all but $\eta n$ vertices will be isomorphic to the $(2R)$-neighbourhood of the root of a $d$-regular tree. Let $Z=Z(R)$ denote the set of these vertices, and $Z_E=Z_E(R)$ be the set of edges whose both endpoints are in $Z$; we have $|Z|=(1\pm \eta)n$ and $|Z_E|=(1\pm 2\eta)\tfrac{d}{2}n$ (since we lose at most $d$ edges for every vertex in $V\backslash Z$).
\item\label{it:expansion} every set $S\subseteq V$ with  $\eta n\leq |S|\leq n/2$ has at least $\zeta |S|$ edges with exactly one endpoint in $S$. 
\end{enumerate} 
Item~\ref{it:neigh} follows by the Azuma-Hoeffding inequality (since for any $R>0$, $\ex[n-Z]=O(1)$ and adding or removing a single edge of $\G$ can change the $(2R)$-neighbourhoods of at most $d^R$ vertices), whereas Item~\ref{it:expansion} follows from a standard union-bound argument. For the rest of the proof, fix $G$ to be any $d$-regular graph satisfying Items~\ref{it:neigh} and~\ref{it:expansion}. 

Consider the graph after the percolation step (ii), i.e., the graph $(V,\hat{E})$. For $v\in Z$, let $\mathbf{1}_v$ be the indicator that there is a neighbour $u\in \partial v$ such that  $vu\in \hat{E}$ and  $u$ has a simple path of length $R$ that starts from it and does not include $v$; if $\mathbf{1}_v=1$, we will say that $v$ belongs to a large component. Since $v\in Z$, it has $d$-neighbours whose $R$-neighbourhoods look like trees, so  $\Erw[\mathbf{1}_v]=1-(1-\hat{p}+\hat{p}\phi_R)^d=\chi_R$. For an edge $e\in Z_E$, let $\mathbf{1}_e$ be the indicator that $e\in \hat{E}$  and that there is a simple path of length $R$ starting from either of the endpoints of $e$ which does not include $e$. Since $e\in Z_E$, we have $\Erw[\mathbf{1}_e]=\hat{p}(1-\phi_R^2)$.  If $\mathbf{1}_e=1$, we will say that $e$ belongs to a large component.   By Azuma's inequality, the random variables $X=\sum_{v\in Z}\mathbf{1}_v$  and $Y=\sum_{e\in Z_E}\mathbf{1}_e$ are within $\eta n$ from their expectation with probability $1-\emm^{-\Omega(n)}$. We have $\Erw[X]=(1\pm \eta )n\chi_R=(\chi\pm 2\eta )n$ and $\Erw[Y]=(1\pm 2\eta )n\psi_R=(\psi\pm 3\eta )n$. Therefore, after percolation step (ii), with probability $1-\emm^{-\Omega(n)}$, we have a set $V_{L}$ of  vertices from $Z$ and a set $E_L$ of edges from $Z_E$ which belong to large components, with $|V_L|=X=(\chi\pm 4\eta)n$ and $|E_L|=Y=(\psi\pm 4\eta)n$. We also conclude that there are at most $n/R$  components with size $\geq R$, which we denote by $\mathcal{C}_1,\hdots,\mathcal{C}_k$ for some $k\leq n/R$. 

Now consider the graph after the percolation step (iii), i.e., the graph $(V,E_\eps)$. We claim that with probability $1-\emm^{-\Omega(n)}$, every partition of  $\mathcal{C}_1,\hdots,\mathcal{C}_k$ into two parts $A,B$ with $|A|,|B|\geq \eta n$ has a path joining them. Indeed, by Menger's theorem and the expansion property in Item~\ref{it:expansion}, for any disjoint vertex sets $A,B$ with $|A|,|B|\geq \eta n$, there are at least $\zeta\eta n$ edge-disjoint paths from $A$ to $B$. Of these paths, at least half of them have at most $\tfrac{d}{\zeta\eta}$ edges (otherwise $|E_{\eps}|>\tfrac{1}{2}dn)$, so the probability that none of them is present after the percolation step is at most $(1-\eps^{d/(\zeta\eta)})^{\zeta\eta n/2}$. Since $k\leq n/R$, there are at most $2^{2n/R}$ ways to partition $\mathcal{C}_1,\hdots,\mathcal{C}_k$ into $A,B$, so by a union bound the probability that a partition exists is upper bounded by $2^{2n/R}(1-\eps^{d/(\zeta\eta)})^{\zeta\eta n/2}\leq \emm^{-\Omega(n)}$ by choosing $R$ large with respect to $\eps,\eta,\zeta$. 

It follows from the above that the final graph $(V, \hat{E}\cup E_{\eps})$ contains a connected component $\cC$ with at least $(\chi-6\eta)n$ vertices from $Z$; otherwise, for the first $i$ such that $|\cC_1\cup \cdots \cC_i|\geq \eta n$, we must have $|\cC_1\cup \cdots \cC_i|\leq (\chi-5\eta) n$, and from $|\cC_1\cup \cdots \cC_k|\geq (\chi-4\eta)n$, we obtain  two disconnected parts $A,B$ with $|A|,|B|\geq \eta n$. This component $\cC$ must contain at least $Y-10d\eta n\geq (\psi- 14d\eta)n$ edges since we lose at most $d$ edges per vertex of $V_L\backslash \cC$.

Note that the vertices in $V(\cC)\cap Z$ belong to $V_L$ and therefore $|V(\cC)\cap Z|\leq |V_L|\leq (\chi+4\eta) n$. There can be at most $\eta n$ vertices in $V(\cC)\backslash Z$ (by Item~\ref{it:neigh}). Therefore $|\cC|=(\chi\pm 6\eta)n$. Similarly, the edges in $E(\cC)\cap Z_E$ belong to $E_L$ and analogously to above we obtain that $|E(\cC)|=(\psi\pm 14d\eta)n$. 

It only remains to show that $\cC$ is the largest component with probability $1-\emm^{\Omega(n)}$. For large $K$, an Azuma-Hoeffding bound shows that the number of vertices that belong to components of size $\leq K$ in the graph  $(V, \hat{E}\cup E_{\eps})$ is at least $(1-\chi-4\eta)n$ with probability $1-\emm^{-\Omega(n)}$. Therefore, via a union bound, we obtain that every component in $(V, \hat{E}\cup E_{\eps})$ other than $\cC$ has at most $20\eta n$ vertices with probability $1-\emm^{\Omega(n)}$, and therefore is  smaller than $\cC$.

This finishes the proof of Lemma~\ref{lem:giantperc}.
\end{proof}

\end{appendix}

\end{document}